\documentclass[a4paper]{article}
\usepackage[authoryear,round]{natbib}
\usepackage{main}
\usepackage[left=1.in, right=1.in, bottom=1.5in, top=1.5in]{geometry}
\setcitestyle{authoryear,round,citesep={;},aysep={,},yysep={;}}
\let\cite\citep

\author{
Daiki Iwade\thanks{
Graduate School of Information Science and Technology, The University of Tokyo, Tokyo, Japan (\href{mailto:iwadedaiki@g.ecc.u-tokyo.ac.jp}{\texttt{iwadedaiki@g.ecc.u-tokyo.ac.jp}})}
\and
Shota Takahashi\thanks{
Graduate School of Information Science and Technology, The University of Tokyo, Tokyo, Japan (\href{mailto:shota@mist.i.u-tokyo.ac.jp}{\texttt{shota@mist.i.u-tokyo.ac.jp}})}
\and
Akiko Takeda\thanks{
Graduate School of Information Science and Technology, The University of Tokyo, Tokyo, Japan (\href{mailto:takeda@mist.i.u-tokyo.ac.jp}{\texttt{takeda@mist.i.u-tokyo.ac.jp}})}~\thanks{Center for Advanced Intelligence Project, RIKEN, Tokyo, Japan}}
\title{Linesearch-Free Nonlinearly Preconditioned GD with Semiautomatic Geometry Design}
\date{\today}
\begin{document}

\maketitle

\begin{abstract}
Nonlinear preconditioning makes it possible to adapt gradient updates to the growth of the objective's curvature, but combining its design with suitable stepsize selection is challenging. We propose \emph{adaptive dual preconditioned gradient descent} (adaptive DPGD), which combines curvature-based preconditioner design with a linesearch-free stepsize rule based on local information. For convex objectives, we establish linear convergence under local strong convexity of the objective. By adapting the stepsize rule, we extend the method to weakly convex objectives and establish asymptotic stationarity without global Lipschitz smoothness. Moreover, our local assumptions give greater freedom in preconditioner design. We exploit this freedom to develop a semiautomatic recipe guided by the objective's curvature, yielding a family of preconditioners that includes those underlying normalized and hyperbolic gradient descent. Experiments on convex and weakly convex problems with real data demonstrate the effectiveness of the preconditioner design and adaptive stepsize selection.
\end{abstract}

\section{Introduction}\label{sec:introduction}

We consider the unconstrained optimization problem
\begin{equation}\label{eq:intro-problem}
    \min_{x\in\RR^d} \quad f(x),
\end{equation}
where $f:\RR^d\to\RR$ is differentiable.
The basic gradient descent (GD) iteration $x_{k+1}=x_k-\gamma_{k+1} \nabla f(x_k)$ follows the Euclidean gradient and scales it by a scalar stepsize.

Nonlinear preconditioning allows the gradient magnitude and, more generally, its direction to be adapted to the objective's curvature.
Dual preconditioned gradient descent (DPGD)~\citep{maddison2021} uses a convex preconditioning function
$\phi:\RR^d\to\RR\cup\{+\infty\}$, minimized at the origin, and updates
\[
    x_{k+1}
    =x_k-\gamma_{k+1}\nabla\phi(\nabla f(x_k)).
\]
The choice of $\phi$ determines the nonlinear geometry of the update.
\citet{oikonomidis2025nonlinearly,oikonomidis2026nonlinearly} have shown that
this framework includes updates underlying normalized gradient
descent (NGD), hyperbolic gradient descent (HGD), and memoryless variants of AdaGrad \citep{adagrad11} and
Adam \citep{kingma2017adammethodstochasticoptimization}.
The DPGD guarantees of~\citet{maddison2021}, however, start from a prescribed
preconditioner and require a global dual relative smoothness condition between $f$ and
$\phi$, while providing limited guidance on how $\phi$ should be designed.
Meanwhile, linesearch-free adaptive stepsizes have been developed for
Euclidean GD using local curvature information
\citep{MM20}, with recent extensions to Bregman settings
\citep{adabregprox2025}.

\begin{table}[tb]
\centering
\caption{Representative methods with convergence guarantees.
NL: nonlinear preconditioning; LS: linesearch stepsize. In the Assumptions columns, Local indicates the required  conditions need only hold on relevant bounded sets. 
cvx: convex; scvx: strongly convex; wcvx: weakly convex; ncvx: nonconvex; Dual for $\nabla f$ denotes relative smoothness/strong convexity in dual space; Anisotropic, smoothness through a preconditioner-dependent nonlinear upper bound on $f$. 
}
\label{tab:combined-comparison}
\begingroup\setlength{\tabcolsep}{2.5pt}
\begin{tabular*}{\textwidth}{@{\extracolsep{\fill}}lclcllc@{}}
\toprule
\multirow{2}{*}{Method} & \multirow{2}{*}{NL} & \multirow{2}{*}{Stepsize} & \multicolumn{3}{c}{Assumptions} & \multirow{2}{*}{Convergence} \\
\cmidrule(lr){4-6}
 & & & Local & class of $f$  & $\nabla f$ smooth & \\
\midrule
GD & \xmarkr & Fixed/LS & \xmarkr/\cmarkg & scvx & Lipschitz & Linear\\
Adaptive GD\textsuperscript{a} & \xmarkr & Adaptive & \cmarkg & scvx & Lipschitz & Linear \\
DPGD\textsuperscript{b} & \cmarkg & Fixed & \xmarkr & Dual scvx & Dual & Linear \\
DPGD-LS\textsuperscript{b} & \cmarkg & LS & \xmarkr & cvx & Dual & $\phi(\nabla f(x_K))-\phi(0)=O(K^{-1})$ \\
AdaPGNC\textsuperscript{c} & \xmarkr & Adaptive & \xmarkr & wcvx & Lipschitz & $\min_{k\le K}\|\nabla f(x_k)\|=O(K^{-1/2})$ \\
HGD\textsuperscript{d} & \cmarkg & Fixed & \xmarkr & ncvx & Anisotropic & $\min_{k\le K}\|\nabla f(x_k)\|=O(K^{-1/2})$ \\
\midrule
\multirow{2}{*}{\textbf{Adaptive DPGD}} & \multirow{2}{*}{\cmarkg} & \multirow{2}{*}{Adaptive} & \cmarkg & scvx & Dual & Linear \\
 & & & \cmarkg & wcvx & Lipschitz & $\nabla f(x_k)\to0$ \\
\bottomrule
\end{tabular*}
\endgroup
\smallskip
\resizebox{\linewidth}{!}{%
\textsuperscript{a}\citet{MM20}; \textsuperscript{b}\citet{maddison2021}; \textsuperscript{c}\citet{ye2025simpleadaptiveproximalgradient}; \textsuperscript{d}\citet{oikonomidis2025nonlinearly}.%
}
\end{table}

Combining these two ideas is nontrivial. Euclidean adaptive GD estimates local curvature directly from changes in gradients and iterates, whereas DPGD nonlinearly transforms the gradient through $\nabla\phi$. Its behavior therefore depends on the interaction between the curvature of $f$ and the geometry induced by $\phi$, rather than on either one alone. Consequently, a direct transfer of Euclidean adaptive rules does not provide sufficient control of the update. Our key observation is that this interaction can instead be quantified through local primal--dual Bregman curvature information. This leads to our central question:
\begin{quote}
\emph{Can nonlinear preconditioners be designed from objective curvature
while their stepsizes are adapted from local information, without global
compatibility constants or repeated linesearch?}
\end{quote}

\paragraph{Contributions.}
We answer this question affirmatively by proposing \emph{adaptive dual preconditioned gradient descent} (adaptive DPGD),
which uses local Bregman curvature information for the stepsize rule. The resulting local convergence conditions provide greater freedom in the choice of $\phi$, which we exploit to construct nonlinear preconditioners systematically from simple curvature models.
\begin{itemize}
    \item \textbf{Adaptive stepsizes for DPGD.}
    We control the growth of the stepsize and adjust it using local curvature estimates expressed through Bregman divergences, which measure deviations from first-order approximations. The stepsize is updated directly, without a linesearch.
    \item \textbf{Convergence analysis.}
    For convex objectives, we establish linear convergence under local strong convexity of $f$ and $\phi$ and local dual relative smoothness, without a global dual relative smoothness constant unlike~\citet{maddison2021}. By adapting the stepsize rule to a weak convexity parameter, we preserve the DPGD update and establish $\phi(\nabla f(x_k))\to\phi(0)$ and $\nabla f(x_k)\to0$ without global Lipschitz smoothness.
    \item \textbf{Semiautomatic preconditioner design.}
    Local dual relative smoothness gives greater freedom in preconditioner design. We develop a semiautomatic construction using the normal and tangential inverse curvatures of radial models of nonquadratic growth. We establish that the resulting families satisfy the preconditioner assumptions of our convergence theory. We recover the preconditioners underlying normalized and hyperbolic gradient descent.
\end{itemize}

 \paragraph{Relation to prior work} Linesearch-free stepsize adaptation based on local curvature was developed for Euclidean GD by \citet{MM20} and extended to Bregman proximal-gradient methods by \citet{adabregprox2025}. Building on the gradient-space specialization of Bregman adaPG, we establish linear convergence for adaptive DPGD under local dual relative smoothness and local strong convexity. \citet{maddison2021} introduced DPGD under global dual relative smoothness, while \citet{oikonomidis2025nonlinearly,oikonomidis2026nonlinearly} developed broader nonlinear-preconditioning frameworks under global generalized or anisotropic smoothness. For weakly convex objectives, the adaptive method of \citet{ye2025simpleadaptiveproximalgradient} also requires global smoothness. Further comparisons are given in \cref{app:preconditioning-related-work}.

\section{Preliminaries}\label{sec:preliminaries}
Additional definitions and auxiliary results are collected in \cref{app:technical-preliminaries}.
When a minimizer exists, write $x_\star\in\amin f$ and $f_\star=f(x_\star)$.
For $h:\RR^d\to\RR\cup\setE{+\infty}$, let $\Dom h$ denote its effective domain and $\Int\Dom h$ its interior. For $a\in\RR$, write $[a]_+:=\max\setE{a,0}$
and adopt the convention $c/0:=+\infty$ for $c>0$.
For open $\Omega\subset\RR^d$ and $k\in\NN$, $C^k(\Omega)$ denotes real-valued functions on $\Omega$ with continuous derivatives up to order $k$.

\begin{definition}[Weak convexity]\label{def:weak-convexity}
   A function $f:\RR^d\to\RR\cup\setE{+\infty}$ is $\eta$-weakly convex, for $\eta\ge0$, if
    $x\longmapsto f(x)+\frac{\eta}{2}\norm{x}^2$
   is convex on $\RR^d$~\citep[see, e.g.,][]{PaquetteWeak18}. The case $\eta=0$ is convexity.
\end{definition}

\begin{notation}\label{def:bregman-divergence}
Let $h:\RR^d\to\RR\cup\{+\infty\}$ be proper and lower semicontinuous, with nonempty $\Omega_h:=\operatorname{int}\operatorname{dom}h$, and suppose $h\in C^1(\Omega_h)$. For $x,y\in\Omega_h$, write
\begin{equation}\label{eq:bregman-divergence}
D_h(x,y):=h(x)-h(y)-\langle\nabla h(y),x-y\rangle,
\qquad
\Delta_h(x,y):=D_h(x,y)+D_h(y,x).
\end{equation}
For $x,y,z\in\Omega_h$, the following three-point identity holds:
\begin{equation}\label{eq:bregman-three-point}
\langle\nabla h(y)-\nabla h(x),z-x\rangle
=D_h(x,y)+D_h(z,x)-D_h(z,y).
\end{equation}
We refer to $D_h$ as the Bregman divergence generated by $h$ and to $\Delta_h$ as its symmetrization. When $h$ is convex, both are nonnegative.
\end{notation}

For a differentiable $f:\RR^d\to\RR$, let $\mathcal P$ denote the class of proper, lower semicontinuous, and convex functions $\phi:\RR^d\to\RR\cup\{+\infty\}$ satisfying $\nabla f(\RR^d)\subset\Omega_\phi:=\Int\Dom\phi$ and $\phi\in C^1(\Omega_\phi)$.
The next definition localizes the global dual relative smoothness condition introduced by \citet[Prop.~3.3]{maddison2021}, allowing the constant to depend on the compact set.
\begin{definition}[Local dual relative smoothness]\label{def:local-dual-relative-smoothness}
    Let $f\in C^1(\RR^d)$ be convex and $\phi\in\mathcal P$. We say that $\phi$ is locally dual relatively smooth with respect to $f$ if, for every nonempty compact set $\mathcal K\subset\RR^d$, there exists $L_{\mathcal K}>0$ such that
    \begin{equation}\label{eq:local-dual-relative-smoothness}
       D_\phi(\nabla f(x),\nabla f(y))
        \le L_{\mathcal K}D_f(y,x),
        \qquad x,y\in\mathcal K.
    \end{equation}
\end{definition}
The corresponding local comparison bounds for the weakly convex setting, together with standard bounds on compact sets used in the convergence proofs, are collected in \cref{app:technical-preliminaries}.

\section{Adaptive DPGD Algorithms}\label{sec:adaptive-dpgd-algorithms}
Fix a known $\eta\ge0$ and set $g_\eta(x):=f(x)+\frac{\eta}{2}\norm{x}^2$, so that $g_0=f$.
The update and stepsize rules use only function values and first derivatives. The following regularity conditions are imposed for the convergence analysis.
\begin{assump}[Objective and preconditioner]\label{assump:main}
    \leavevmode
    \begin{enumerate}
        \item $f\in C^2(\RR^d)$ is level bounded, and $\nabla^2g_\eta(x)\succ0$ for every $x\in\RR^d$.
        \item $\phi\in\mathcal P$ and $\amin\phi=\{0\}$.
        \item If $\eta=0$, $\phi$ is locally dual relatively smooth with respect to $f$. If $\eta>0$, $\nabla^2f$ and $\nabla\phi$ are locally Lipschitz on $\RR^d$ and $\Omega_\phi$, respectively.
    \end{enumerate}
\end{assump}
Condition (i) gives local strong convexity of $g_\eta$ and ensures that $f$ attains its minimum; fix $x_\star\in\amin f$.
For $\eta=0$, level boundedness and existence of a unique minimizer are equivalent under $\nabla^2f\succ0$~\citep[Thm.~2.6 and Prop.~3.23]{Rockafellar1997-de}.
Condition (ii) makes $\nabla\phi$ a natural preconditioner for driving $\nabla f$ toward $0$.
Condition (iii) retains the weaker convex regularity and ensures that $\nabla\phi\circ\nabla f$ is locally Lipschitz and
$D_\phi(\nabla f(x),\nabla f(y))\le L_{\mathcal K}D_{g_\eta}(y,x)$ on each compact $\mathcal K\subset\RR^d$.
These bounds are established in \cref{prop:local-regularity} in \cref{app:technical-preliminaries}.

\subsection{Adaptive DPGD}\label{subsec:adaptive-dpgd}
Define $T_k(x):=x-\gamma_k\nabla\phi(\nabla f(x))$. For positive stepsizes, the DPGD update is
\begin{equation}\label{eq:update}
    \xf=\hf(\x)=\x-\gf\np{\x},
    \qquad \beta_{k+1}:=\frac{\gamma_{k+1}}{\gamma_k}.
\end{equation}
By \cref{assump:main}(ii), $x_k=x_{k-1}$ is equivalent to $\nabla f(x_{k-1})=0$.
We terminate at stationary points and define the following quantities at nonstationary iterations, with $\delta>0$:
\begin{equation}\label{eq:local-estimates}
    \widehat L_k^{\mathrm d,\eta}:=\frac{\Delta_\phi(\nabla f(x_k),\nabla f(x_{k-1}))}{\Delta_{g_\eta}(x_k,x_{k-1})},
    \qquad
    \mathcal R^\eta_{k,\delta}:=\frac{2D_{g_\eta}(x_k+\delta[T_k(x_k)-T_k(x_{k-1})],x_k)}
    {\delta^2\Delta_{g_\eta}(x_k,x_{k-1})},
\end{equation}
where $\Delta_h$ for some $h$ is in \eqref{eq:bregman-divergence}.
Set $\zeta_k^\eta:=\frac{D_{g_\eta}(x_{k-1},x_k)}{D_{g_\eta}(x_k,x_{k-1})}$ and $m_k^\eta:=\frac{\Delta_{g_\eta}(x_k,x_{k-1})}{\norm{x_k-x_{k-1}}^2}$.
Local estimates use $g_\eta$, while updates use $\nabla f$.
The symmetrized comparison bounds $\widehat L_k^{\mathrm d,\eta}$ on compact sets.

Initialize $\beta_0=\beta_1=1$ and $\gamma_1=\gamma_0$.
For $\eta>0$, fix $s\ge0$ and $\beta_{\mathrm{cap}}\in(1,3)$.
For $k\ge0$, define
\begin{equation}\label{eq:stepsize-growth}
    \bar\beta_{k+1}:=
    \begin{cases}
        \sqrt{1+\beta_k},&\text{if }\eta=0,\\
        \min\{1+s\gamma_k,\beta_{\mathrm{cap}}\},&\text{if }\eta>0.
    \end{cases}
\end{equation}
For $k\ge1$, set
\begin{equation}\label{eq:stepsize-rule}
    \beta_{k+1}:=
    \begin{cases}
        \displaystyle\min\Biggl\{
            \bar\beta_{k+1},\;
            \frac{\zeta_k^0}{1+\zeta_k^0}
            \frac{1}{2\bar\beta_{k+1}
                [\mathcal R^0_{k,2\bar\beta_{k+1}}-(1-\gamma_k\widehat L_k^{\mathrm d,0})]_+}
        \Biggr\},&\text{if }\eta=0,\\[3ex]
        \displaystyle\min\Biggl\{
            \bar\beta_{k+1},\;
            \frac{\zeta_k^\eta}{1+\zeta_k^\eta}
            \frac{2(1-\beta_k/\bar\beta_k)+\eta(1+\zeta_k^\eta)/(m_k^\eta\zeta_k^\eta)}
            {[\bar\beta_{k+1}\mathcal R^\eta_{k,\bar\beta_{k+1}}
                -1+\gamma_k\widehat L_k^{\mathrm d,\eta}+\eta/m_k^\eta]_+}
        \Biggr\},&\text{if }\eta>0.
    \end{cases}
\end{equation}
The two cases use different Lyapunov arguments; substituting $\eta=0$ into the positive-$\eta$ formula does not recover the convex rule.
When $\eta>0$ and $s=0$, $\bar\beta_{k+1}=1$ and the stepsizes are nonincreasing.
We present Adaptive DPGD with these stepsizes in \cref{def:adadpgdupdate}.
\begin{algorithm}[!thbp]
    \caption{Adaptive DPGD}\label{def:adadpgdupdate}\label{def:weak-update-rule}
    \DontPrintSemicolon
    \KwIn{$x_0\in\RR^d$, $\gamma_0>0$, $\eta\ge0$; for $\eta>0$: $s\ge0$, $1<\beta_{\mathrm{cap}}<3$}
    Initialize $\beta_0=\beta_1=1$, $\gamma_1=\gamma_0$, and $\bar\beta_1$ by \eqref{eq:stepsize-growth} with $k=0$\;
    $x_1\gets x_0-\gamma_1\np{x_0}$\;
    \For{$k=1,2,\ldots$}{
        Compute \eqref{eq:local-estimates} and $\zeta_k^\eta$; if $\eta>0$, also compute $m_k^\eta$\;
        Set $\beta_{k+1}$ using \eqref{eq:stepsize-growth}--\eqref{eq:stepsize-rule}\;
        $\gamma_{k+1}\gets\beta_{k+1}\gamma_k$\;
        $x_{k+1}\gets x_k-\gamma_{k+1}\np{x_k}$\;
    }
\end{algorithm}

\section{Convergence Analysis}\label{sec:convergence-analysis}
If $x_k=x_{k-1}$, then $x_k$ is stationary, \ie, $\nabla f(x_k)=0$, and, when $\eta=0$, the unique minimizer. Unless stated otherwise, we consider an infinite nonstationary sequence with $x_k\neq x_{k-1}$ for every $k\ge1$.
We present the basic identities, Lyapunov inequalities, and principal convergence results. Complete proofs are given in \cref{app:technical-preliminaries,app:convex-convergence-proofs,app:weak-convergence-proofs}.

We first establish two-step identities shared by the convex and weakly convex analyses.
Define the gap from the optimal value of the preconditioning function by $G_k:=\phi(\nabla f(x_k))-\phi(0)$.
By \cref{assump:main}(ii), $G_k\ge0$. Define the cross term
\begin{equation}\label{eq:def-cross-term}
    \chi_{k+1}:=\beta_{k+1}
    \inpr{T_k(x_{k-1})-T_k(x_k)}{\nabla f(x_k)-\nabla f(x_{k+1})}.
\end{equation}

These identities hold without convexity in both settings; see \cref{app:common-two-step-proof} for a proof.

\begin{lemma}[Basic two-step identities]\label{lem:basic-two-step-identities}
   Let $\setE{x_k}$ be generated by \eqref{eq:update} with positive stepsizes $\setE{\gamma_{k+1}}$.
   Then, for every $k\ge1$,
    \begin{equation}\label{eq:common-gap-identity}
        G_{k-1}-G_k
        =\frac{\Delta_f(x_k,x_{k-1})-\gamma_k\Delta_\phi(\nabla f(x_k),\nabla f(x_{k-1}))}{\gamma_k}
        +D_\phi(\nabla f(x_{k-1}),\nabla f(x_k)).
    \end{equation}
    Furthermore, if $\bar x\in\RR^d$ satisfies $\nabla f(\bar x)=0$, then
    \begin{equation}\label{eq:common-distance-identity}
        D_f(x_{k+1},\bar x)+\gamma_{k+1}G_k+D_f(x_k,x_{k+1})
        =D_f(x_k,\bar x)+\chi_{k+1}-\gamma_{k+1}D_\phi(0,\nabla f(x_k)).
    \end{equation}
\end{lemma}

\subsection{Convex objectives}\label{subsec:convex-convergence}
Throughout this subsection, $\eta=0$; we omit the superscript $0$ from the local quantities.
We combine the two identities in \cref{lem:basic-two-step-identities} and bound the cross term $\chi_{k+1}$ using the Bregman three-point identity~\eqref{eq:bregman-three-point}. This yields a decrease inequality for the Lyapunov function $\U$ defined below. The proof is given in \cref{app:proof-convex-lyapunov}.

\begin{lemma}[Lyapunov inequality]\label{lem:lyap}
Define the Lyapunov function by
    \begin{equation}\label{eq:convex-lyapunov}
        \U:=D_f(\x,\xo)+\ga(1+\rhh)\Pb+\paren*{1-\frac{\rh}{2\rhh}}D_f(\xb,\x).
    \end{equation}
Under the update rule in \cref{def:adadpgdupdate} and \cref{assump:main} with $\eta=0$, for every $k\ge1$,
    \begin{equation}\label{eq:convex-lyapunov-descent}
        \Uf\le\U-\ga(1+\rhh-\rf\rhf)\Pb\le\U.
    \end{equation}
\end{lemma}

This Lyapunov inequality yields the following main result.
\begin{theorem}\label{thm:main}
   Suppose that \cref{assump:main} holds with $\eta=0$. Then the sequence $\setE{\x}$ generated by \cref{def:adadpgdupdate} satisfies
    $\x\to\xo$ and $\nf(\x)\to0$.
\end{theorem}
\begin{proof}[Proof sketch of \cref{thm:main}]
    \cref{lem:lyap} gives boundedness of the iterates. On a compact set containing the sequence, $\widehat L_k^{\mathrm d}$ is uniformly bounded above, $\zeta_k$ is bounded away from $0$, and $\mathcal R_{k,2\bar\beta_{k+1}}\to1$ if $\gamma_k\to0$. These local estimates and the stepsize rule imply $\sum_k\gamma_k=\infty$. Combining the existence of a subsequence approaching stationarity with an argument about the limit of the Lyapunov function proves convergence of the entire sequence to $x_\star$. The complete proof is given in \cref{app:proof-convex-convergence}.
\end{proof}

For the convergence rate, we make an additional assumption.
\begin{assump}\label{assump:rate}
For every nonempty compact convex set $\mathcal K^*\subset\Omega_\phi$, there exists $\mu_{\mathcal K^*}>0$ such that $\phi$ is $\mu_{\mathcal K^*}$-strongly convex on $\mathcal K^*$.
\end{assump}
We will discuss the construction of such preconditioners in \cref{sec:preconditioner-design}, which shows how to construct preconditioners satisfying \cref{assump:rate}.

\begin{theorem}[R-linear convergence]\label{thm:bestiteraterate}

Suppose that \cref{assump:main,assump:rate} hold with $\eta=0$. Then there exist $C_f,C_\nabla>0$ and $\alpha\in(0,1)$ such that the sequence generated by \cref{def:adadpgdupdate} with the convex stepsize rule satisfies, for all $k\ge0$,
\begin{equation}\label{eq:linear-last-iterate}
    0\le f(x_k)-f_\star\le C_f \alpha^k,
    \qquad
    \norm{\nabla f(x_k)}\le C_\nabla \alpha^k.
\end{equation}

\end{theorem}

\begin{proof}[Proof sketch of \cref{thm:bestiteraterate}]
We do not require a positive lower bound on the one-step Lyapunov decrease coefficient. Local bounds and \cref{assump:rate} bound the stepsizes above and away from zero and give $\mathcal E_k\le C_EG_{k-1}$ for some $C_E>0$. Summing over a fixed number $N$ of iterations yields $\mathcal E_{k+N}\le\kappa\mathcal E_k$ with $\kappa\in(0,1)$, implying \eqref{eq:linear-last-iterate}. The complete proof is in \cref{app:proof-best-iterate-rate}.
\end{proof}

\subsection{Weakly convex objectives}\label{subsec:weak-convergence}
Throughout this subsection, \cref{assump:main} holds with $\eta>0$. We use the local comparison bound \eqref{eq:local-comparison}, established in \cref{prop:local-regularity} in \cref{app:technical-preliminaries}, with the convexified function $g_\eta$.
Since $x_\star$ globally minimizes $f$ on $\RR^d$, $\nabla f(x_\star)=0$ and
\begin{equation}\label{eq:weak-nonnegative-gaps}
   D_f(x_k,x_\star)
    =f(x_k)-f(x_\star)\ge0.
\end{equation}
Moreover, the definition of $g_\eta$ gives
\begin{equation}\label{eq:weak-sym-difference}
    \Delta_f(x_k,x_{k-1})
    =
    \Delta_{g_\eta}(x_k,x_{k-1})
    -\eta\norm{x_k-x_{k-1}}^2
\end{equation}
and, using \eqref{eq:local-estimates} and $m_k^\eta$,
\begin{equation}
    \Delta_f(x_k,x_{k-1})
    -\gamma_k\Delta_\phi(\nabla f(x_k),\nabla f(x_{k-1}))
    =
    \left(
        1-\gamma_k \widehat L_k^{\mathrm d,\eta}-\frac{\eta}{m_k^\eta}
    \right)
    \Delta_{g_\eta}(x_k,x_{k-1}).
    \label{eq:weak-curvature-rewrite}
\end{equation}

We next establish the Lyapunov inequality for weakly convex objectives. The proof is given in \cref{app:proof-weak-lyapunov}.
\begin{lemma}[Lyapunov inequality]\label{lem:weak-lyapunov}
   Define the Lyapunov function by
    \begin{equation}\label{eq:weak-lyapunov}
        \mathcal{E}_k^{\mathrm{weak}}
        :=
       D_f(x_k,x_\star)
        +\frac32\gamma_kG_{k-1}
        +\left(
            1-\frac{\beta_k}{\bar\beta_k}
        \right)D_{g_\eta}(x_{k-1},x_k)
        +\frac{\eta}{2}\norm{x_k-x_{k-1}}^2.
    \end{equation}
Under \cref{def:weak-update-rule} with $\eta>0$, for every $k\ge1$,
    \begin{equation}\label{eq:weak-lyapunov-descent}
        \mathcal{E}_{k+1}^{\mathrm{weak}}
        \le
        \mathcal{E}_k^{\mathrm{weak}}
        -\frac{\gamma_k}{2}(3-\beta_{k+1})G_{k-1}.
    \end{equation}
Since $\beta_{k+1}\le\beta_{\mathrm{cap}}<3$, the coefficient of $\gamma_kG_{k-1}$ in the decrease is uniformly positive.

\end{lemma}

The Lyapunov inequality yields the following stationarity result.
\begin{theorem}[Stationarity for weakly convex objectives]\label{thm:weak-stationarity}
   Under \cref{assump:main} with $\eta>0$ and \cref{def:weak-update-rule},
    \begin{equation}
        \phi(\nabla f(x_k))\to\phi(0). \qquad
        \text{Consequently,}\quad \nabla f(x_k)\to0\label{eq:weak-gradient-stationarity}.
    \end{equation}

\end{theorem}
\begin{proof}[Proof sketch of \cref{thm:weak-stationarity}]
    \cref{lem:weak-lyapunov} gives boundedness of the iterates. If $\gamma_k\to0$, local estimates yield $\gamma_{k+1}\ge\gamma_k(1-C\gamma_k)$ eventually for some $C>0$, so $\sum_k\gamma_k=\infty$ in all cases. Combining this with $\sum_k\gamma_kG_{k-1}<\infty$ gives $\liminf_kG_k=0$. Local Lipschitz continuity of $\phi\circ\nabla f$ then yields $G_k\to0$ and \eqref{eq:weak-gradient-stationarity}. The complete proof is in \cref{app:proof-weak-stationarity}.
\end{proof}

\section{Semiautomatic Design of Preconditioners}\label{sec:preconditioner-design}
We construct nonlinear preconditioners from simple curvature models of the objective. Unlike fixed-stepsize DPGD analyses~\citep{maddison2021} that require a problem-specific global compatibility condition for a prescribed preconditioner, our local framework allows us to design $\phi$ from the curvature behavior that we wish to correct. The construction is semiautomatic: once a target model, a curvature direction, and an interpolation parameter are selected, the preconditioner is determined.

\paragraph{Design principle.}
For convex $f$, the ideal relation
$\nabla^2\phi(\nabla f(x))\approx(\nabla^2 f(x))^{-1}$
would make the Jacobian of the preconditioned direction,
$\nabla^2\phi(\nabla f(x))\nabla^2f(x)$, close to the identity.
We therefore design radial preconditioners that nonlinearly rescale gradients to compensate for nonquadratic curvature growth.
We model the growth to be corrected by a radial
target $t(x)=\tau(\|x\|)$, where $\tau:[0,\infty)\to\mathbb{R}$ is $C^2$
on $(0,\infty)$ with $\tau'(r)>0$ and $\tau''(r)>0$ for $r>0$. For $r=\|x\|>0$,
$\he t(x)=\tau''(r)\frac{xx^\top}{r^2}+\frac{\tau'(r)}{r}\paren*{I-\frac{xx^\top}{r^2}}$.
Thus the normal and tangential curvatures are  $\tau''(r)$ and $\tau'(r)/r$, respectively. We select the corresponding inverse curvature  $h_{\mathrm{N}}=(\tau''(r))^{-1}$ or $h_{\mathrm{T}}=(\tau'(r)/r)^{-1}$ and express it as a function of the gradient magnitude $\varrho:=\|\nabla t(x)\|$.

\paragraph{Interpolation and preconditioner families.}
Suppose that the selected inverse curvature, expressed as a function of $\varrho$, has a leading power-law factor $\varrho^a$ for some $a\in\mathbb{R}$. The exponent $a$ captures the asymptotic inverse-curvature growth of the target model. To preserve
unit curvature near the origin while matching this behavior for large gradients, define, for $b>0$,
\begin{equation}\label{eq:power-interpolation}
    \psi_{a,b}(\varrho):=(1+\varrho^b)^{a/b},
    \qquad \varrho\ge0.
\end{equation}
Then $\psi_{a,b}(0)=1$ and $\psi_{a,b}(\varrho)/\varrho^a\to1$ as $\varrho\to\infty$. Hence the resulting geometry is GD-like for small gradients and matches
the selected power-law factor for large gradients. The exponent $a$ is determined by the target curvature, whereas $b$ controls the transition.

We next construct a radial preconditioner $\phi(y)=\nu(\|y\|)$ whose normal or tangential curvature matches $\psi_{a,b}$. Since these curvatures are $\nu''(\varrho)$ and $\nu'(\varrho)/\varrho$, respectively, integration gives
\begin{equation}\label{eq:normal-construction}
    \nu_{\mathrm{N}}^{a,b}(\varrho):=\int_0^\varrho(\varrho-s)\psi_{a,b}(s)\,ds,
    \qquad \nu_{\mathrm{T}}^{a,b}(\varrho):=\int_0^\varrho s\psi_{a,b}(s)\,ds,
\end{equation}
and we set $\phi_{\mathrm{N}}^{a,b}(y):=\nu_{\mathrm{N}}^{a,b}(\norm{y}), \ \phi_{\mathrm{T}}^{a,b}(y):=\nu_{\mathrm{T}}^{a,b}(\norm{y})$. Thus $\phi_{\mathrm{N}}^{a,b}$ matches $\psi_{a,b}$ in its normal curvature, while $\phi_{\mathrm{T}}^{a,b}$ matches it in its tangential curvature. For the tangential family $\phi_{\mathrm{T}}^{a,b}$, we restrict to $a\ge -1$.

For weakly convex objectives, the Hessian of $f$ need not be positive definite, so the inverse-Hessian interpretation above is no longer literal. We instead use the radial target as a model of the curvature behavior to be corrected, while applying the same preconditioner construction.

\begin{table}[!htb]
\centering
\setlength{\tabcolsep}{5pt}
\renewcommand{\arraystretch}{1.15}
\caption{Preconditioning functions and related methods, where $q=p/(p-1)>1$.}
\label{tab:preconditioner-recipes}
\begin{tabularx}{\linewidth}{@{}>{\raggedright\arraybackslash}p{0.55\linewidth}>{\raggedright\arraybackslash\hyphenpenalty=10000}X@{}}
\toprule
Preconditioning function & Relation \\
\midrule
$\displaystyle\phi_{\mathrm{N}}^{0,b}(y)=\phi_{\mathrm{T}}^{0,b}(y)=\frac{\norm{y}^2}{2}\quad(b>0)$ & GD \\
\addlinespace[7pt]
$\displaystyle\phi_{\mathrm{N}}^{q-2,1}(y)=\frac{(1+\norm{y})^q-1-q\norm{y}}{q(q-1)}$ & New variant \\
\addlinespace[7pt]
$\displaystyle\phi_{\mathrm{T}}^{q-2,2}(y)=\frac{(1+\norm{y}^2)^{q/2}-1}{q}$ & $p$-norm DPGD~\citep{maddison2021} \\
\addlinespace[7pt]
$\displaystyle\phi_{\mathrm{N}}^{-1,1}(y)=(1+\norm{y})\log(1+\norm{y})-\norm{y}$ & \citet[Table~1]{oikonomidis2025nonlinearly} \\
\addlinespace[7pt]
$\displaystyle\phi_{\mathrm{N}}^{-1,2}(y)=\norm{y}\operatorname{arsinh}(\norm{y})-\sqrt{1+\norm{y}^2}+1$ & HGD~\citep{oikonomidis2025nonlinearly,oikonomidis2026nonlinearly} \\
\addlinespace[7pt]
$\displaystyle\phi_{\mathrm{T}}^{-1,1}(y)=\phi_{\mathrm{N}}^{-2,1}(y)=\norm{y}-\log(1+\norm{y})$ & NGD; exponential-penalty DPGD~\citep{Zhang2020Why,maddison2021} \\
\addlinespace[7pt]
$\displaystyle\phi_{\mathrm{T}}^{-1,2}(y)=\sqrt{1+\norm{y}^2}-1$ & AdaGrad-type~\citep{adagrad11} \\
\addlinespace[7pt]
$\displaystyle\lim_{b\to\infty}\phi_{\mathrm{T}}^{-1,b}(y)=\begin{cases}\norm{y}^2/2,&\norm{y}\le1,\\\norm{y}-1/2,&\norm{y}>1\end{cases}$ & Gradient clipping~\citep{Zhang2020Why} \\
\bottomrule
\end{tabularx}
\end{table}

\paragraph{Choosing the exponent from the target model.}
The exponent $a$ follows directly from the target curvature.
For the power-law target $\tau(r)=r^p/p$ with $p>1$, let $q:=p/(p-1)$. Since $\varrho=\tau'(r)=r^{p-1}$, both inverse curvatures are proportional to $\varrho^{q-2}$, giving $a=q-2=(2-p)/(p-1)$; for example, $p=6$ gives $a=-0.8$. For the exponential target $\tau(r)=e^r$, we have $h_{\mathrm{N}}(\varrho)=\varrho^{-1}$. Thus the normal construction gives $a=-1$. Further details are given in \cref{subsubsec:principal-targets}.

\cref{tab:preconditioner-recipes} lists selected preconditioning functions and related methods. Further examples are given in \cref{app:preconditioner-design-supplement}. The family also highlights different large-gradient behaviors. The AdaGrad-type member \(\phi_{\mathrm{T}}^{-1,2}\) asymptotically normalizes the gradient, whereas the new power-law variant \(\phi_{\mathrm{N}}^{q-2,1}\) produces a preconditioned gradient with norm of order \(\|y\|^{q-1}\), with \(q-1=1/(p-1)\). Thus the latter adapts the scaling of the preconditioned gradient to the curvature growth of the target model.

We finally verify that the constructed preconditioner families satisfy
the regularity conditions required by the convergence analyses in
\cref{sec:adaptive-dpgd-algorithms,sec:convergence-analysis}. We omit the superscripts when $a,b$ are fixed.
The proof is given in \cref{app:preconditioner-proofs}.

\begin{proposition}\label{prop:power-interpolation-preconditioner}
Let $a\in\RR$ and $b>0$, and let $\phi=\phi_{\mathrm{N}}$ or, when $a\ge-1$, $\phi=\phi_{\mathrm{T}}$. Then:
\begin{enumerate}
    \item $\phi\in C^2(\RR^d)$, $\Omega_\phi=\RR^d$, $\nabla^2\phi\succ0$, $\nabla^2\phi(0)=I$, and $\amin\phi=\{0\}$. Thus \cref{assump:main}(ii) holds and $\nabla\phi$ is locally Lipschitz.
    \item If $f$ satisfies \cref{assump:main}(i) with $\eta=0$, then \cref{assump:main,assump:rate} hold.
    \item If $f$ satisfies \cref{assump:main}(i) with $\eta>0$ and $\nabla^2f$ is locally Lipschitz, \cref{assump:main} holds.
\end{enumerate}
\end{proposition}

\section{Numerical Experiments}\label{sec:numerical-experiments}
We test curvature-based design and adaptive stepsizes on convex and weakly convex problems. We use Adaptive DPGD with $\phi_{\mathrm N}^{a,1}$ (ours), which performed well across the tested problems, with $a$ determined by each target. We use a linesearch only to initialize the first stepsize for Adaptive DPGD; all subsequent steps are linesearch-free. We compare it with Euclidean Adaptive DPGD, with $\phi(y)=\norm{y}^2/2$, to assess the effect of nonlinear preconditioning, and with GD Armijo~\citep{armijo66} and, for convex problems, adaptive GD of \citet{MM20}. Additional preconditioner examples and convergence comparisons are given in \cref{app:numerical-additional-results}. We use adaptive GD with $\lambda_0=10^{-10}$ and $\theta_0=+\infty$ as in \citet{MM20}, and weakly convex Adaptive DPGD with $s=0.1$ and $\beta_{\mathrm{cap}}=2$. Linesearches use dyadic trials, with Armijo coefficient $10^{-4}$. Initial line-search candidates are $1$ for $\ell_p$ loss and weakly convex problems. For powered hinge loss, the initial inverse stepsize for DPGD is $\|\nabla\phi(\nabla f(x_0))\|$, and the initial Armijo stepsize is $1/\|\nabla f(x_0)\|$.

\cref{fig:numerical-main} plots the relative objective gap $(f(x_k)-f_{\rm ref})/(f(x_0)-f_{\rm ref})$ for convex problems and the relative gradient norm $\norm{\nabla f(x_k)}/\norm{\nabla f(x_0)}$ for weakly convex problems, against runtime. Reference values, timing, and the computational environment are detailed in \cref{app:numerical-environment}. Runs stop when the relative gradient norm is at most $10^{-6}$ for five consecutive iterations. For the convex comparisons in \cref{fig:numerical-main}, we use nonzero initial points to examine convergence from regions of pronounced nonquadratic curvature growth. Preprocessing and initial points are given in \cref{app:dataset-preprocessing}. Full plots, including gradient evaluations, are in \cref{app:main-full-comparisons}.

\paragraph{Convex problems.}
For data rows $b_i^\top$ of $B\in\RR^{n\times d}$, responses $y\in\RR^n$, and $p>2$, we use the regularized $\ell_p$ loss: $\min_{x\in\RR^d}\frac{1}{np}\sum_{i=1}^n\abs{b_i^\top x-y_i}^p +\frac{1}{2n}\norm{x}^2$. As in DPGD for $p$-norm regression~\citep{maddison2021}, the quadratic term ensures a positive definite Hessian. We use $\tau(r)=r^p/p$ (\cref{subsubsec:principal-targets}).
We set $p=6$ on \texttt{covtype}~\citep{covertype,LIBSVM} ($n=581012$, $d=53$). Our method reduces the objective gap faster than the three baselines (\cref{fig:numerical-main}(a)).

For binary labels $y_i\in\{-1,1\}$ and feature rows $a_i^\top$, we also consider powered hinge loss~\citep{Chen2004SoftMargin}: $\min_{x\in\RR^d}\frac{1}{np}\sum_{i=1}^n[1-y_i a_i^\top x]_+^p+\frac{1}{2n}\norm{x}^2$, with $p=6$ and the target $\tau(r)=r^p/p$.
On \texttt{a9a}~\citep{LIBSVM,adult} ($n=32561$, $d=123$), our method reaches a relative objective gap of $10^{-6}$ sooner than the three baselines (\cref{fig:numerical-main}(b)).
We use high-dimensional TF--IDF data for \texttt{LEDGAR} ($n=70000$, $d=19996$) and \texttt{SCOTUS} ($n=6400$, $d=126405$), distributed by LIBSVM~\citep{LIBSVM}, with the most frequent class versus the rest. Our method gives a rapid initial reduction of the objective gap and reaches the gradient stopping criterion sooner than the three baselines on both datasets (\cref{fig:numerical-main}(c),(d)).

\paragraph{Weakly convex problems.}
For data rows $a_i^\top$ of $A\in\RR^{n\times d}$ and labels $y\in\setE{-1,1}^n$, we combine exponential loss~\citep{Friedman2000Boosting} and a Cauchy penalty~\citep{Karakus2020Cauchy}: $\min_{x\in\RR^d} \frac{1}{n}\sum_{i=1}^n\exp(-y_i a_i^\top x) +\omega\sum_{j\in\mathcal J}\log(1+x_j^2/\sigma^2)$. We exclude the intercept from $\mathcal J$, set $\omega=0.1$, $\sigma=1$, and use $\eta=\omega/(4\sigma^2)$ and $\tau(r)=\exp(r)$. The global weak convexity bound is derived in \cref{app:numerical-weak-exp-cauchy-details}.
On \texttt{Skin Non-Skin}~\citep{skin_segmentation} ($n=245057$, $d=4$) and \texttt{Rice Cammeo/Osmancik}~\citep{Cinar_Koklu_2019} ($n=3810$, $d=8$), our method reduces the gradient norm faster than its Euclidean variant and GD Armijo (\cref{fig:numerical-main}(e),(f)). Across the three methods, terminal objective values agree within $5\times10^{-12}$ on both datasets.

\paragraph{Effects of adaptive stepsizes and preconditioner design.}
The normal constructions provide effective alternatives to existing preconditioners, with competitive runtime even for the numerically evaluated $\phi_{\mathrm N}^{a,2}$ (\cref{app:numerical-observations}). On \texttt{covtype}, the powered hinge datasets, and the weakly convex problems, Adaptive DPGD reaches the stopping criterion with fewer gradient evaluations and less runtime than its linesearch counterparts in nearly all comparisons, illustrating the benefit of avoiding repeated trial evaluations. For example, on \texttt{covtype} with $\phi_{\mathrm N}^{a,1}$, Adaptive DPGD uses 94 gradient evaluations and 2.25 seconds, versus 194 evaluations and 3.10 seconds for DPGD linesearch.

Varying only $a$ in $\phi_{\mathrm N}^{a,1}$, the model-derived exponents $a=-0.8$ on \texttt{svmguide1} and $a=-1$ on \texttt{Skin Non-Skin} give faster convergence than the alternative exponents in \cref{fig:exponent-ablation-n1}. This directly supports choosing $a$ from the curvature model rather than treating it as a tunable hyperparameter.

\section{Conclusion}

We proposed Adaptive DPGD, combining curvature-based nonlinear preconditioner design with linesearch-free stepsizes. Without specifying a global dual relative smoothness constant, we prove linear convergence of the objective gap and gradient norm under local strong convexity in convex problems, and stationarity with the same explicit update in weakly convex problems. Real-data experiments demonstrate the benefits of both preconditioner design and adaptive stepsizes.

The framework suggests several natural extensions. For weakly convex problems, it would be interesting to obtain quantitative stationarity rates and to adapt to an unknown weak-convexity parameter while retaining the explicit DPGD update. On the design side, the normal/tangential choice and interpolation parameter could be selected automatically from objective or iterate information. The same ideas also motivate coordinatewise or blockwise curvature estimates for separable preconditioners related to AdaGrad and SNGD/SHGD~\citep{adagrad11,Zhang2020Why,oikonomidis2025nonlinearly,oikonomidis2026nonlinearly} as discussed in \cref{app:preconditioner-design-supplement}.

\begin{figure}[!tp]
    \centering
    \includegraphics[width=\linewidth]{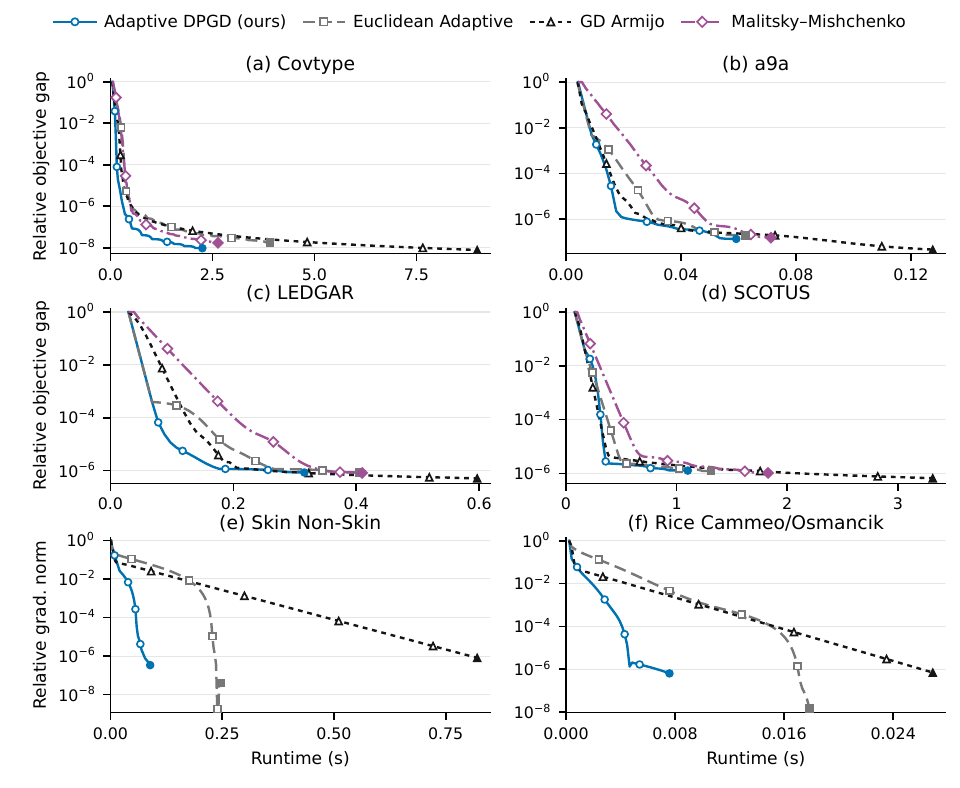}
    \caption{Runtime comparisons. Adaptive DPGD (ours) uses $\phi_{\mathrm N}^{a,1}$. (a) $\ell_6$ loss on \texttt{covtype}; (b--d) powered hinge loss ($p=6$) on \texttt{a9a}, \texttt{LEDGAR}, and \texttt{SCOTUS}; (e--f) exponential loss with Cauchy penalty on \texttt{Skin Non-Skin} and \texttt{Rice Cammeo/Osmancik}. Convex panels show relative objective gaps; weakly convex panels show relative gradient norms.}
    \label{fig:numerical-main}
\end{figure}

\begin{figure}[!tp]
    \centering
    \includegraphics[width=0.96\linewidth]{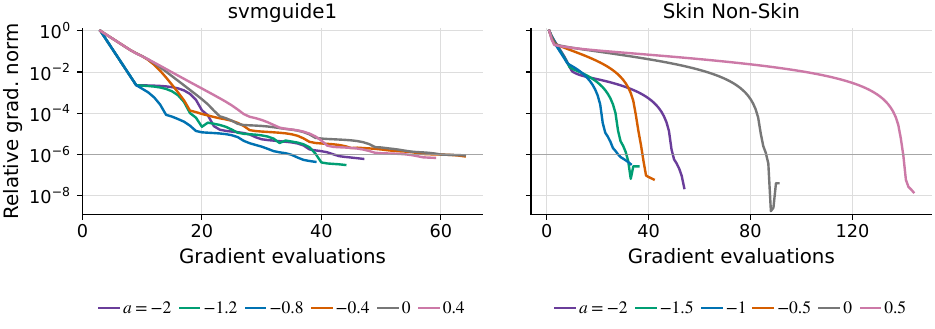}
    \caption{Adaptive DPGD with $\phi_{\mathrm N}^{a,1}$, varying only $a$. Model exponents: $-0.8$ (powered hinge loss, $p=6$, \texttt{svmguide1}) and $-1$ (\texttt{Skin Non-Skin}); $a=0$ is Euclidean Adaptive DPGD.}
    \label{fig:exponent-ablation-n1}
\end{figure}

\subsubsection*{Acknowledgments}
This project has been funded by the Japan Society for the Promotion of Science (JSPS); JSPS KAKENHI Grant Number JP23K28041 and JP25K21156.

\newpage

\bibliographystyle{plainnat}
\bibliography{reference}

@article{maddison2021,
  author        = {Maddison, Chris J. and Paulin, Daniel and Teh, Yee Whye and Doucet, Arnaud},
  title         = {Dual Space Preconditioning for Gradient Descent},
  journal       = {SIAM Journal on Optimization},
  volume        = {31},
  number        = {1},
  pages         = {991--1016},
  year          = {2021},
}

@article{adabregprox2025,
  title = {Linesearch-Free Adaptive {Bregman} Proximal Gradient for Convex Minimization Under Local Relative Smoothness},
  volume = {210},
  number = {3},
  journal = {Journal of Optimization Theory and Applications},
  publisher = {Springer Science and Business Media LLC},
  author = {Ou,  Hongjia and Latafat,  Puya and Themelis,  Andreas},
  year = {2026},
  eid = {68},
}

@inproceedings{MM20,
  title         = {Adaptive Gradient Descent without Descent},
  author        = {Malitsky, Yura and Mishchenko, Konstantin},
  booktitle     = {Proceedings of the 37th International Conference on Machine Learning},
  pages         = {6702--6712},
  year          = {2020},
  volume        = {119},
  series        = {Proceedings of Machine Learning Research},
  publisher     = {PMLR},
}

@book{beck1stoder,
  author        = {Beck, Amir},
  title         = {First-Order Methods in Optimization},
  publisher     = {Society for Industrial and Applied Mathematics},
  year          = {2017},
  address       = {Philadelphia, PA},
}

@article{armijo66,
  author        = {Larry Armijo},
  title         = {Minimization of functions having {Lipschitz} continuous first partial derivatives},
  volume        = {16},
  journal       = {Pacific Journal of Mathematics},
  number        = {1},
  publisher     = {Pacific Journal of Mathematics, A Non-profit Corporation},
  pages         = {1--3},
  year          = {1966},
}

@article{adagrad11,
  author        = {John Duchi and Elad Hazan and Yoram Singer},
  title         = {Adaptive Subgradient Methods for Online Learning and Stochastic Optimization},
  journal       = {Journal of Machine Learning Research},
  year          = {2011},
  volume        = {12},
  number        = {61},
  pages         = {2121--2159},
}

@inproceedings{oikonomidis2026nonlinearly,
  author        = {Oikonomidis, Konstantinos and Quan, Jan and Patrinos, Panagiotis},
  booktitle     = {Advances in Neural Information Processing Systems},
  pages         = {38957--38988},
  publisher     = {Curran Associates, Inc.},
  title         = {Nonlinearly Preconditioned Gradient Methods: Momentum and Stochastic Analysis},
  volume        = {38},
  year          = {2025},
}

@inproceedings{oikonomidis2025nonlinearly,
  title         = {Nonlinearly Preconditioned Gradient Methods under Generalized Smoothness},
  author        = {Oikonomidis, Konstantinos and Quan, Jan and Laude, Emanuel and Patrinos, Panagiotis},
  booktitle     = {Proceedings of the 42nd International Conference on Machine Learning},
  pages         = {47132--47154},
  year          = {2025},
  volume        = {267},
  series        = {Proceedings of Machine Learning Research},
  publisher     = {PMLR},
}

@inproceedings{Zhang2020Why,
  title         = {Why Gradient Clipping Accelerates Training: A Theoretical Justification for Adaptivity},
  author        = {Jingzhao Zhang and Tianxing He and Suvrit Sra and Ali Jadbabaie},
  booktitle     = {International Conference on Learning Representations},
  year          = {2020},
}

@inproceedings{kingma2017adammethodstochasticoptimization,
  author       = {Diederik P. Kingma and
                  Jimmy Ba},
  title        = {{Adam}: {A} Method for Stochastic Optimization},
  booktitle    = {International Conference on Learning Representations},
  year         = {2015},
  bibsource    = {dblp computer science bibliography, https://dblp.org}
}

@book{nesterov2004,
  author        = {Nesterov, Yurii},
  title         = {Introductory Lectures on Convex Optimization: A Basic Course},
  series        = {Applied Optimization},
  volume        = {87},
  publisher     = {Kluwer Academic Publishers},
  address       = {Boston, MA},
  year          = {2004},
}

@book{Rockafellar1997-de,
  title         = {Variational Analysis},
  author        = {Rockafellar, Ralph Tyrrell and Wets, Roger J.-B.},
  publisher     = {Springer},
  series        = {Grundlehren der mathematischen Wissenschaften},
  edition       = {1st},
  year          = {1997},
  address       = {Berlin, Germany},
  language      = {en},
}

@inproceedings{PaquetteWeak18,
  title         = {Catalyst for Gradient-based Nonconvex Optimization},
  author        = {Paquette, Courtney and Lin, Hongzhou and Drusvyatskiy, Dmitriy and Mairal, Julien and Harchaoui, Zaid},
  booktitle     = {Proceedings of the Twenty-First International Conference on Artificial Intelligence and Statistics},
  pages         = {613--622},
  year          = {2018},
  volume        = {84},
  series        = {Proceedings of Machine Learning Research},
  publisher     = {PMLR},
}

@article{yagishita2025simplelinesearchfreefirstordermethods,
  title         = {Simple linesearch-free first-order methods for nonconvex optimization},
  author        = {Shotaro Yagishita and Masaru Ito},
  year          = {2025},
  eprint        = {2509.14670},
  archiveprefix = {arXiv},
  primaryclass  = {math.OC},
  journal       = {arXiv preprint arXiv:2509.14670},
}

@article{PatelDiminishing24,
  author        = {Patel, Vivak and Berahas, Albert S.},
  title         = {Gradient Descent in the Absence of Global {Lipschitz} Continuity of the Gradients},
  journal       = {SIAM Journal on Mathematics of Data Science},
  volume        = {6},
  number        = {3},
  pages         = {602--626},
  year          = {2024},
}

@article{ye2025simpleadaptiveproximalgradient,
  title         = {A Simple Adaptive Proximal Gradient Method for Nonconvex Optimization},
  author        = {Zilong Ye and Shiqian Ma and Junfeng Yang and Danqing Zhou},
  year          = {2025},
  eprint        = {2510.06079},
  archiveprefix = {arXiv},
  primaryclass  = {math.OC},
  journal       = {arXiv preprint arXiv:2510.06079},
}

@article{kin8nmOpenML,
  title         = {{OpenML}: Networked Science in Machine Learning},
  author        = {Joaquin Vanschoren and van Rijn, Jan N. and Bernd Bischl and Luis Torgo},
  journal       = {ACM SIGKDD Explorations Newsletter},
  volume        = {15},
  year          = {2013},
  pages         = {49--60},
}

@article{LIBSVM,
  author        = {Chang, Chih-Chung and Lin, Chih-Jen},
  title         = {{LIBSVM}: A library for support vector machines},
  year          = {2011},
  issue_date    = {April 2011},
  publisher     = {Association for Computing Machinery},
  address       = {New York, NY, USA},
  volume        = {2},
  number        = {3},
  journal       = {ACM Transactions on Intelligent Systems and Technology},
  numpages      = {27},
  eid = {27},
}

@misc{adult,
  author        = {Becker, Barry and Kohavi, Ronny},
  title         = {{Adult}},
  year          = {1996},
  howpublished  = {UCI Machine Learning Repository},
}

@misc{covertype,
  author        = {Blackard, Jock},
  title         = {{Covertype}},
  year          = {1998},
  howpublished  = {UCI Machine Learning Repository},
}

@misc{skin_segmentation,
  author        = {Bhatt, Rajen and Dhall, Abhinav},
  title         = {{Skin Segmentation}},
  year          = {2009},
  howpublished  = {UCI Machine Learning Repository},
}

@misc{HoaiThai2024,
  author        = {Pham Thi Hoai and Nguyen Pham Duy Thai},
  title         = {Composite Optimization Problems via Novel Proximal Gradient Algorithms and Applications},
  year          = {2024},
  howpublished  = {Optimization Online},
}

@article{xuChenAAMD2026,
  title         = {Adaptive Accelerated Mirror Descent in Primal and Dual Spaces},
  author        = {Zeyi Xu and Long Chen},
  year          = {2026},
  eprint        = {2606.02787},
  archiveprefix = {arXiv},
  primaryclass  = {math.OC},
  journal       = {arXiv preprint arXiv:2606.02787},
}

@article{Friedman2000Boosting,
  title         = {Additive logistic regression: A statistical view of boosting (With discussion and a rejoinder by the authors)},
  volume        = {28},
  number        = {2},
  journal       = {The Annals of Statistics},
  publisher     = {Institute of Mathematical Statistics},
  author        = {Friedman, Jerome and Hastie, Trevor and Tibshirani, Robert},
  year          = {2000},
  pages         = {337--407},
}

@article{Chen2004SoftMargin,
  author        = {Di-Rong Chen and Qiang Wu and Yiming Ying and Ding-Xuan Zhou},
  title         = {Support Vector Machine Soft Margin Classifiers: Error Analysis},
  journal       = {Journal of Machine Learning Research},
  volume        = {5},
  pages         = {1143--1175},
  year          = {2004},
}

@article{Karakus2020Cauchy,
  author        = {Karaku\c{s}, Oktay and Mayo, Perla and Achim, Alin},
  journal       = {IEEE Transactions on Signal Processing},
  title         = {Convergence Guarantees for Non-Convex Optimisation With {Cauchy}-Based Penalties},
  year          = {2020},
  volume        = {68},
  pages         = {6159--6170},
}

@article{Takahashi2026-tp,
  title         = {Adaptive conditional gradient sliding: Projection-free and line-search-free acceleration},
  author        = {Takahashi, Shota},
  journal       = {arXiv preprint arXiv:2601.20443},
  year          = {2026},
  archiveprefix = {arXiv},
  primaryclass  = {math.OC},
  eprint        = {2601.20443},
}

@article{Li2025-cb,
  author = {Li, Tianjiao and Lan, Guanghui},
  title = {A simple uniformly optimal method without line search for convex optimization},
  journal = {Mathematical Programming},
  volume = {219},
  number = {1},
  year = {2026},
  pages = {219--256},
}

@article{Suh2025-eq,
  title         = {An adaptive and parameter-free {Nesterov}'s accelerated gradient method for convex optimization},
  author        = {Suh, Jaewook J. and Ma, Shiqian},
  journal       = {arXiv preprint arXiv:2505.11670},
  year          = {2025},
  archiveprefix = {arXiv},
  primaryclass  = {math.OC},
  eprint        = {2505.11670},
}

@inproceedings{Park2026-cd,
  title         = {Adaptive gradient descent on {Riemannian} manifolds and its applications to {Gaussian} variational inference},
  author        = {Park, Jiyoung and Suh, Jaewook J. and Wang, Bofan and Bhattacharya, Anirban and Ma, Shiqian},
  booktitle     = {International Conference on Learning Representations},
  year          = {2026},
}

@inproceedings{Takahashi2026-ch,
  title         = {Fast {Frank--Wolfe} Algorithms with Adaptive {Bregman} Step-Size for Weakly Convex Functions},
  author        = {Takahashi, Shota and Pokutta, Sebastian and Takeda, Akiko},
  booktitle     = {International Conference on Learning Representations},
  year          = {2026},
}

@article{Takahashi2026-rv,
  title = {Majorization-Minimization {Bregman} Proximal Gradient Algorithms for {NMF} with the {Kullback–Leibler} Divergence},
  volume = {208},
  number = {1},
  journal = {Journal of Optimization Theory and Applications},
  publisher = {Springer Science and Business Media LLC},
  author = {Takahashi,  Shota and Tanaka,  Mirai and Ikeda,  Shiro},
  year = {2025},
  eid = {14},
}

@article{Takahashi2024-ej,
  author = {Takahashi, Shota and Takeda, Akiko},
  title = {Approximate {Bregman} proximal gradient algorithm for relatively smooth nonconvex optimization},
  journal = {Computational Optimization and Applications},
  volume = {90},
  number = {1},
  year = {2025},
  pages = {227--256},
}

@article{Takahashi2022-ml,
  title         = {New {Bregman} proximal type algorithms for solving {DC} optimization problems},
  author        = {Takahashi, Shota and Fukuda, Mituhiro and Tanaka, Mirai},
  journal       = {Computational Optimization and Applications},
  publisher     = {Springer Science and Business Media LLC},
  volume        = {83},
  number        = {3},
  pages         = {893--931},
  year          = {2022},
}

@article{Bolte2018-zt,
  title         = {First order methods beyond convexity and {Lipschitz} gradient continuity with applications to quadratic inverse problems},
  author        = {Bolte, J\'{e}r\^{o}me and Sabach, Shoham and Teboulle, Marc and Vaisbourd, Yakov},
  journal       = {SIAM Journal on Optimization},
  publisher     = {Society for Industrial \& Applied Mathematics (SIAM)},
  volume        = {28},
  number        = {3},
  pages         = {2131--2151},
  year          = {2018},
}

@article{Bauschke2017-hg,
  title         = {A descent lemma beyond {Lipschitz} gradient continuity: First-order methods revisited and applications},
  author        = {Bauschke, Heinz H. and Bolte, J\'{e}r\^{o}me and Teboulle, Marc},
  journal       = {Mathematics of Operations Research},
  publisher     = {Institute for Operations Research and the Management Sciences (INFORMS)},
  volume        = {42},
  number        = {2},
  pages         = {330--348},
  year          = {2017},
}

@book{Nemirovski1983-yw,
  title         = {{Problem Complexity and Method Efficiency in Optimization}},
  author        = {Nemirovski, Arkadi S. and Yudin, David B.},
  publisher     = {Wiley},
  address       = {Chichester},
  series        = {Wiley Series in Discrete Mathematics},
  year          = {1983},
}

@inproceedings{Mukkamala2019-mk,
  title         = {Beyond alternating updates for matrix factorization with inertial {Bregman} proximal gradient algorithms},
  author        = {Mukkamala, Mahesh Chandra and Ochs, Peter},
  booktitle     = {Advances in Neural Information Processing Systems},
  publisher     = {Curran Associates, Inc.},
  volume        = {32},
  pages         = {4268--4278},
  year          = {2019},
}

@article{Fujiki2025-do,
  title         = {Approximate {Bregman} proximal gradient algorithm with variable metric {Armijo--Wolfe} line search},
  author        = {Fujiki, Kiwamu and Takahashi, Shota and Takeda, Akiko},
  journal       = {arXiv preprint arXiv:2510.06615},
  year          = {2025},
  archiveprefix = {arXiv},
  primaryclass  = {math.OC},
  eprint        = {2510.06615},
}

@inproceedings{chalkidis-etal-2022-lexglue,
title = "{L}ex{GLUE}: A Benchmark Dataset for Legal Language Understanding in {E}nglish",
author = "Chalkidis, Ilias  and
Jana, Abhik  and
Hartung, Dirk  and
Bommarito, Michael  and
Androutsopoulos, Ion  and
Katz, Daniel  and
Aletras, Nikolaos",
booktitle = "Proceedings of the 60th Annual Meeting of the Association for Computational Linguistics (Volume 1: Long Papers)",
year = "2022",
publisher = "Association for Computational Linguistics",
pages = "4310--4330",
}

@inproceedings{tuggener-etal-2020-ledgar,
title = "{LEDGAR}: A Large-Scale Multi-label Corpus for Text Classification of Legal Provisions in Contracts",
author = {Tuggener, Don  and
von D{\"a}niken, Pius  and
Peetz, Thomas  and
Cieliebak, Mark},
booktitle = "Proceedings of the Twelfth Language Resources and Evaluation Conference",
year = "2020",
publisher = "European Language Resources Association",
pages = "1235--1241",
language = "eng",
}

@article{Cinar_Koklu_2019,
  title = {Classification of Rice Varieties Using Artificial Intelligence Methods},
  volume = {7},
  number = {3},
  journal = {International Journal of Intelligent Systems and Applications in Engineering},
  author = {Cinar, Ilkay and Koklu, Murat},
  year = {2019},
  pages = {188--194},
}

\appendix
\section{Technical Preliminaries}\label[appendix]{app:technical-preliminaries}
We collect local regularity definitions, auxiliary results on compact sets, and the proof of the two-step identities used in both convergence analyses.

\subsection{Definitions and preliminary results}

\begin{definition}[Level boundedness]\label{def:level-bounded}
   A function $h$ is level bounded if, for every $c\in\RR$,
    \begin{equation}
        \setE{x\in\RR^d\mid h(x)\le c}
    \end{equation}
   is bounded; see, e.g., \citet[Definition~1.8]{Rockafellar1997-de}.
\end{definition}
\begin{definition}[Local Lipschitz continuity]\label{def:local-regularity}
   A map $S$ on an open set $\Omega\subset\RR^d$ is locally Lipschitz continuous if, for each $x\in\Omega$, there exist a neighborhood $U_x\subset\Omega$ and a constant $L_x>0$ such that, for every $y,z\in U_x$,
    \begin{equation}
        \norm{S(y)-S(z)}\le L_x\norm{y-z}.
    \end{equation}
The pointwise property is called strict continuity in \citet[Definition~9.1(b)]{Rockafellar1997-de}.
\end{definition}

\begin{definition}[Strong convexity on a set and local strong convexity]\label{def:local-strong-convexity}
   A differentiable convex function $h$ is $\mu_{\mathcal K}$-strongly convex on a convex set $\mathcal K\subset\Int\Dom h$ if $\mu_{\mathcal K}>0$ and, for all $x,y\in\mathcal K$,
    \begin{equation}\label{eq:strong-convex-on-set}
       h(y)
        \ge
       h(x)+\inpr{\nabla h(x)}{y-x}
        +\frac{\mu_{\mathcal K}}{2}\norm{y-x}^2.
    \end{equation}
The first-order condition above is the standard characterization of strong convexity restricted to $\mathcal K$; see, e.g., \citet[Theorem~5.24]{beck1stoder}.
We call $h$ locally strongly convex if it is strongly convex on every nonempty compact convex set $\mathcal K\subset\Int\Dom h$.
\end{definition}

\begin{lemma}\label{lem:local-strong-convexity}
   Let $h\in C^2(\Omega)$ be convex, with $\nabla^2h(x)\succ0$ for all $x\in\Omega$.
   Then $h$ is locally strongly convex on $\Omega$. In particular, for every nonempty compact convex set $\mathcal K\subset\Omega$, there exist $m_{\mathcal K},M_{\mathcal K}>0$ such that
    \begin{equation}\label{eq:local-hessian-bounds}
       m_{\mathcal K}I\preceq\nabla^2h(x)\preceq M_{\mathcal K}I,
        \qquad x\in\mathcal K.
    \end{equation}
Furthermore, for every $x,y\in\mathcal K$,
    \begin{equation}\label{eq:local-bregman-quadratic-bounds}
        \frac{m_{\mathcal K}}{2}\norm{x-y}^2
        \le D_h(x,y)
        \le \frac{M_{\mathcal K}}{2}\norm{x-y}^2
    \end{equation}
   and
    \begin{equation}\label{eq:local-symmetric-gradient-bounds}
       m_{\mathcal K}\norm{x-y}^2
        \le \Delta_h(x,y)
        \le M_{\mathcal K}\norm{x-y}^2,
        \qquad
       m_{\mathcal K}\norm{x-y}
        \le \norm{\nabla h(x)-\nabla h(y)}
        \le M_{\mathcal K}\norm{x-y}.
    \end{equation}
In particular, if $x\neq y$, then
    \begin{equation}\label{eq:local-bregman-ratio-bound}
        \frac{m_{\mathcal K}}{M_{\mathcal K}}
        \le
        \frac{D_h(x,y)}{D_h(y,x)}
        \le
        \frac{M_{\mathcal K}}{m_{\mathcal K}}.
    \end{equation}
\end{lemma}
\begin{proof}[Proof of \cref{lem:local-strong-convexity}]
Fix a nonempty compact convex set $\mathcal K\subset\Omega$.
Continuity of $\nabla^2h$ and compactness of $\mathcal K$ imply that $x\mapsto\lambda_{\min}(\nabla^2h(x))$ has a positive minimum on $\mathcal K$, and $x\mapsto\lambda_{\max}(\nabla^2h(x))$ has a finite maximum.
Denoting these by $m_{\mathcal K},M_{\mathcal K}>0$ gives \eqref{eq:local-hessian-bounds}.

Fix $x,y\in\mathcal K$ and $w:=x-y$. Convexity of $\mathcal K$ gives $y+tw\in\mathcal K$ for $0\le t\le1$, and Taylor's integral formula yields
\begin{equation}
   D_h(x,y)
    =
    \int_0^1(1-t)
    \inpr{\nabla^2h(y+tw)w}{w}\,dt.
\end{equation}
Integrating \eqref{eq:local-hessian-bounds} gives \eqref{eq:local-bregman-quadratic-bounds}. Similarly,
\begin{equation}
    \Delta_h(x,y)
    =
    \inpr{\nabla h(x)-\nabla h(y)}{w}
    =
    \int_0^1
    \inpr{\nabla^2h(y+tw)w}{w}\,dt
\end{equation}
gives the bounds on $\Delta_h$ in \eqref{eq:local-symmetric-gradient-bounds}. Moreover, the argument of \citet[Theorem~5.12]{beck1stoder}, applied along segments in $\mathcal K$, uses
\begin{equation}
    \nabla h(x)-\nabla h(y)
    =
    \int_0^1\nabla^2h(y+tw)w\,dt
\end{equation}
together with \eqref{eq:local-hessian-bounds} to obtain $\norm{\nabla h(x)-\nabla h(y)}\le M_{\mathcal K}\norm{w}$. Conversely, Cauchy--Schwarz and $\Delta_h(x,y)\ge m_{\mathcal K}\norm{w}^2$ give $\norm{\nabla h(x)-\nabla h(y)}\ge m_{\mathcal K}\norm{w}$.
Equation~\eqref{eq:strong-convex-on-set} is equivalent to the lower bound in \eqref{eq:local-bregman-quadratic-bounds}. Finally, for $x\neq y$, apply \eqref{eq:local-bregman-quadratic-bounds} to $(x,y)$ and $(y,x)$ and take ratios to obtain \eqref{eq:local-bregman-ratio-bound}.
\end{proof}
The next lemma gives local Lipschitz continuity of the composition used in the convex convergence analysis.
\begin{lemma}\label{lem:locallip}
   Let $f:\RR^d\to\RR$ be $C^2$, with $\nabla^2f(x)\succ0$ for every $x\in\RR^d$.
   Suppose that $\phi\in\mathcal P$ is locally dual relatively smooth with respect to $f$.
   Then $\nabla\phi\circ\nabla f$ is locally Lipschitz continuous on $\RR^d$.
\end{lemma}
\begin{proof}[Proof of \cref{lem:locallip}]
Fix $\bar{x}\in\RR^d$. 
Positive definiteness of $\he f(\bar{x})$, continuity of $\he f$, and the inverse function theorem \citep[Corollary~9.55]{Rockafellar1997-de} provide a bounded open convex neighborhood $\mathcal{V}$ of $\bar{x}$, an open convex neighborhood $\mathcal{W}$ of $\nf(\bar{x})$,
and constants $m,M>0$ such that
\begin{equation}
\overline{\mathcal{V}}\ \text{is compact},
\qquad
\mathcal{W}\subset\nf(\mathcal{V})\cap\Omega_\phi
\end{equation}
and
\begin{equation}\label{eq:local-bilip-gradient}
m\norm{x-y}
\le
\norm{\nf(x)-\nf(y)}
\le
M\norm{x-y},
\qquad
x,y\in\mathcal{V}.
\end{equation}
Moreover, $\he f(x)\preceq MI$ on $\mathcal{V}$ gives
\begin{equation}\label{eq:local-bregman-upper}
D_f(y,x)
\le
\frac{M}{2}\norm{y-x}^2,
\qquad
x,y\in\mathcal{V}.
\end{equation}

For $u,v\in\mathcal{W}$, set
\begin{equation}
    x:=\paren*{\nf|_{\mathcal{V}}}^{-1}(u),
    \qquad
    y:=\paren*{\nf|_{\mathcal{V}}}^{-1}(v).
\end{equation}
By \eqref{eq:local-dual-relative-smoothness}, \eqref{eq:local-bilip-gradient}, and \eqref{eq:local-bregman-upper},
\begin{align}
    &D_\phi(u,v) \le L_{\overline{\mathcal{V}}}D_f(y,x) \le \frac{L_{\overline{\mathcal{V}}}M}{2} \norm{y-x}^2 \le \frac{L_{\overline{\mathcal{V}}}M}{2m^2} \norm{u-v}^2. \label{eq:local-phi-smoothness}
\end{align}
Set $C:=L_{\overline{\mathcal V}}M/m^2>0$.
We localize the argument of \citet[Theorem~2.1.5]{nesterov2004}.
Continuity of $\nabla\phi$ allows us to choose an open ball $\mathcal U$ about $\nf(\bar{x})$, contained in $\mathcal W$, such that, for all $u,v\in\mathcal U$, the points $v-\frac{\nabla\phi(v)-\nabla\phi(u)}{C}$ lie in $\mathcal W$.
Convexity, expansion of the Bregman divergence, and \eqref{eq:local-phi-smoothness} give
\begin{equation}
    0\le D_\phi\!\left(v-\frac{\nabla\phi(v)-\nabla\phi(u)}{C},u\right)\le D_\phi(v,u)-\frac{1}{2C}\norm{\nabla\phi(v)-\nabla\phi(u)}^2.
\end{equation}
Interchanging $u,v$ and adding yields
\[
    \norm{\nabla\phi(v)-\nabla\phi(u)}^2\le C\Delta_\phi(u,v)\le C\norm{\nabla\phi(v)-\nabla\phi(u)}\norm{u-v}.
\]
Thus $\nabla\phi$ is $C$-Lipschitz continuous on $\mathcal U$.

By \eqref{eq:local-bilip-gradient}, $\nf$ is Lipschitz continuous on $\mathcal{V}$. Consequently, $\nabla\phi\circ\nf$ is Lipschitz continuous on $\mathcal V\cap(\nabla f)^{-1}(\mathcal U)$, a neighborhood of $\bar x$ by continuity of $\nabla f$.
Since $\bar{x}\in\RR^d$ was arbitrary, this composition is locally Lipschitz continuous on $\RR^d$.

\end{proof}

The two regularity conditions in \cref{assump:main}(iii) yield the following common bounds.
\begin{proposition}\label{prop:local-regularity}
    Let $\eta\ge0$ and $f\in C^2(\RR^d)$ with $\nabla^2g_\eta\succ0$ on $\RR^d$. Suppose that $\phi\in\mathcal P$ and \cref{assump:main}(iii) holds. Then the following hold.
    \begin{enumerate}
        \item $g_\eta$ is locally strongly convex.
        \item $\nabla\phi\circ\nabla f$ is locally Lipschitz on $\RR^d$.
        \item For every nonempty compact set $\mathcal K\subset\RR^d$, there exists $L_{\mathcal K}>0$ such that
        \begin{equation}\label{eq:local-comparison}
            D_\phi(\nabla f(x),\nabla f(y))
            \le L_{\mathcal K}D_{g_\eta}(y,x),
            \qquad x,y\in\mathcal K.
        \end{equation}
        Consequently,
        \begin{equation}\label{eq:local-comparison-sym}
            \Delta_\phi(\nabla f(x),\nabla f(y))
            \le L_{\mathcal K}\Delta_{g_\eta}(x,y),
            \qquad x,y\in\mathcal K.
        \end{equation}
    \end{enumerate}
\end{proposition}
\begin{proof}[Proof of \cref{prop:local-regularity}]
Part (i) follows from \cref{lem:local-strong-convexity} applied to $g_\eta$.
If $\eta=0$, part (ii) follows from \cref{lem:locallip}, while \eqref{eq:local-comparison} is the assumed local dual relative smoothness, since $g_0=f$.

If $\eta>0$, $\nabla f$ is locally Lipschitz because $f\in C^2$, and $\nabla\phi$ is locally Lipschitz on $\Omega_\phi$ by \cref{assump:main}(iii).
Their composition is therefore locally Lipschitz~\citep[Exercise~9.8(c)]{Rockafellar1997-de}.
Fix a nonempty compact $\mathcal K\subset\RR^d$.
Choose a Lipschitz constant $L_f$ of $\nabla f$ on a compact convex set containing $\mathcal K$, and a Lipschitz constant $L_\phi$ of $\nabla\phi$ on a compact convex subset of $\Omega_\phi$ containing $\nabla f(\mathcal K)$~\citep[Theorem~9.2]{Rockafellar1997-de}.
Part (i) gives $m_g>0$ such that $\nabla^2g_\eta\succeq m_g I$ on a compact convex set containing $\mathcal K$.
The descent lemma~\citep[Lemma~5.7]{beck1stoder} yields
\begin{align}
    D_\phi(\nabla f(x),\nabla f(y))
    &\le \frac{L_\phi}{2}\norm{\nabla f(x)-\nabla f(y)}^2
    \le \frac{L_\phi L_f^2}{2}\norm{x-y}^2
    \le \frac{L_\phi L_f^2}{m_g}D_{g_\eta}(y,x).
\end{align}
This proves \eqref{eq:local-comparison}.
In either case, adding it with $x,y$ interchanged gives \eqref{eq:local-comparison-sym}.
\end{proof}

For $\eta=0$, \eqref{eq:local-comparison} is local dual relative smoothness as defined in \cref{def:local-dual-relative-smoothness}.
For $\eta>0$, its left-hand side still uses $\nabla f$, while the right-hand side uses the convexified function $g_\eta$; it is a consequence of the regularity assumptions rather than an additional hypothesis.
The locally Lipschitz Hessian is used separately in the weakly convex stepsize analysis.

\subsection{\texorpdfstring{Proof of \cref{lem:basic-two-step-identities}}{Proof of the common two-step identities}}\label[appendix]{app:common-two-step-proof}
\begin{proof}[Proof of \cref{lem:basic-two-step-identities}]
We adapt the argument of \citet[Lemma~3.1]{adabregprox2025} to DPGD.
   Since \eqref{eq:update} gives $T_k(x_{k-1})=x_k$,
    \begin{equation}
       T_k(x_{k-1})-T_k(x_k)
        =\gamma_k\nabla\phi(\nabla f(x_k)).
    \end{equation}
Expanding $D_\phi(\nabla f(x_{k-1}),\nabla f(x_k))$, we obtain
    \begin{align}
       D_\phi(\nabla f(x_{k-1}),\nabla f(x_k))
        &=
        \phi(\nabla f(x_{k-1}))-\phi(\nabla f(x_k))
        \\
        &\quad
        -\frac{1}{\gamma_k}
        \inpr{T_k(x_{k-1})-T_k(x_k)}
        {\nabla f(x_{k-1})-\nabla f(x_k)}.
    \end{align}
   By the definition of $T_k$, the last inner product is
    \begin{align}
        &\inpr{T_k(x_{k-1})-T_k(x_k)}
        {\nabla f(x_{k-1})-\nabla f(x_k)}
        \\
        &\qquad=
        \Delta_f(x_k,x_{k-1})
        -\gamma_k\Delta_\phi(\nabla f(x_k),\nabla f(x_{k-1})).
    \end{align}
Combining these expressions and using $G_k$ gives \eqref{eq:common-gap-identity}.

   Next, suppose that $\nabla f(\bar x)=0$. By \eqref{eq:update},
    \begin{equation}
        0=
        \inpr{\nabla\phi(\nabla f(x_k))}{\nabla f(\bar x)-\nabla f(x_{k+1})}
        -\frac{1}{\gamma_{k+1}}
        \inpr{x_k-x_{k+1}}{\nabla f(\bar x)-\nabla f(x_{k+1})}.
    \end{equation}
For the first term, \eqref{eq:def-cross-term} and $\gamma_{k+1}=\beta_{k+1}\gamma_k$ give
    \begin{align}
        \inpr{\nabla\phi(\nabla f(x_k))}{\nabla f(\bar x)-\nabla f(x_{k+1})}
        &=-G_k-D_\phi(0,\nabla f(x_k))
        +\frac{1}{\gamma_{k+1}}\chi_{k+1}.
    \end{align}
   The algebraic three-point identity \eqref{eq:bregman-three-point} also gives
    \begin{align}
        -\inpr{x_k-x_{k+1}}{\nabla f(\bar x)-\nabla f(x_{k+1})}
        &=D_f(x_k,\bar x)-D_f(x_{k+1},\bar x)-D_f(x_k,x_{k+1}).
    \end{align}
   Substituting these expressions and rearranging yields \eqref{eq:common-distance-identity}.
\end{proof}

\section{Proofs for the Convex Analysis}\label[appendix]{app:convex-convergence-proofs}
Throughout this appendix, $\eta=0$ and the algorithm uses \eqref{eq:stepsize-rule}. We omit the superscript $0$ from the local quantities, as in \cref{subsec:convex-convergence}.

\subsection{Proof of the convex Lyapunov inequality}\label[appendix]{app:proof-convex-lyapunov}
Write $\rhm:=(1+\sqrt5)/2$. Since $\beta_0=1$, induction using the update rule gives
\begin{equation}\label{ineq:beta-bound}
    0<\beta_k\le\bar\beta_k,
    \qquad
    1<\bar\beta_k\le\rhm,
    \qquad k\ge1.
\end{equation}

The definition of $\zeta_k$ gives
\begin{equation}
    \sdf(\x,\xb)
    =
    \frac{1+\zeta_k}{\zeta_k}D_f(\xb,\x).
\end{equation}

\begin{proof}[Proof of \cref{lem:lyap}]
We adapt the argument of \citet[Lemmas~3.2 and~3.6]{adabregprox2025} to DPGD.
   Multiply \eqref{eq:common-gap-identity} in \cref{lem:basic-two-step-identities} by $\gamma_{k+1}\bar\beta_{k+1}$ and
   apply \eqref{eq:common-distance-identity} with $\bar x=\xo$.
   Convexity of $\phi$ gives
    \begin{equation}
       D_\phi(0,\nabla f(x_k))\ge0,
        \qquad
       D_\phi(\nabla f(x_{k-1}),\nabla f(x_k))\ge0,
    \end{equation}
    so
    \begin{align}
        &D_f(x_{k+1},\xo)
        +\gamma_{k+1}(1+\bar\beta_{k+1})G_k
        +D_f(x_k,x_{k+1})
        \\
        &\quad\le
       D_f(x_k,\xo)
        +\gamma_{k+1}\bar\beta_{k+1}G_{k-1}
        -\beta_{k+1}\bar\beta_{k+1}(1-\gamma_k \widehat L_k^{\mathrm d})\Delta_f(x_k,x_{k-1})
        +\chi_{k+1}.
        \label{eq:convex-combined-two-step}
    \end{align}

   Applying \eqref{eq:bregman-three-point} with $h=f$, $x=x_k$, $y=x_{k+1}$, and $z=x_k-2\bar\beta_{k+1}[T_k(x_{k-1})-T_k(x_k)]$, and using convexity of $f$, gives
    \begin{align}
    &\chi_{k+1} \le \frac{\beta_{k+1}}{2\bar\beta_{k+1}}D_f(x_k,x_{k+1}) +\beta_{k+1}\bar\beta_{k+1} \mathcal R_{k,2\bar\beta_{k+1}}\Delta_f(x_k,x_{k-1}). \label{eq:convex-cross-term-bound}
\end{align}
   Substituting \eqref{eq:convex-cross-term-bound} into \eqref{eq:convex-combined-two-step} yields
    \begin{align}
        &D_f(x_{k+1},\xo)
        +\gamma_{k+1}(1+\bar\beta_{k+1})G_k
        +\paren*{1-\frac{\beta_{k+1}}{2\bar\beta_{k+1}}}D_f(x_k,x_{k+1})
        \\
        &\quad\le
       D_f(x_k,\xo)
        +\gamma_{k+1}\bar\beta_{k+1}G_{k-1}
        \\
        &\qquad
        +\beta_{k+1}\bar\beta_{k+1}
        \left(
            \mathcal R_{k,2\bar\beta_{k+1}}-(1-\gamma_k \widehat L_k^{\mathrm d})
        \right)
        \Delta_f(x_k,x_{k-1}).
        \label{eq:lyapintermediate}
    \end{align}

   Meanwhile, \eqref{eq:stepsize-rule} ensures
    \begin{equation}\label{eq:convex-stepsize-absorption}
        \beta_{k+1}\bar\beta_{k+1}
        \left[
            \mathcal R_{k,2\bar\beta_{k+1}}-(1-\gamma_k \widehat L_k^{\mathrm d})
        \right]_+
        \frac{1+\zeta_k}{\zeta_k}
        \le\frac12.
    \end{equation}
    Using $\Delta_f(x_k,x_{k-1})=\frac{1+\zeta_k}{\zeta_k}D_f(x_{k-1},x_k)$, $a\le[a]_+$, and $\beta_k\le\bar\beta_k$, \eqref{eq:convex-stepsize-absorption} gives
    \begin{align}
        &\beta_{k+1}\bar\beta_{k+1}
        \left(
            \mathcal R_{k,2\bar\beta_{k+1}}-(1-\gamma_k \widehat L_k^{\mathrm d})
        \right)
        \Delta_f(x_k,x_{k-1})
        \\
        &\quad\le
        \beta_{k+1}\bar\beta_{k+1}
        \left[
            \mathcal R_{k,2\bar\beta_{k+1}}-(1-\gamma_k \widehat L_k^{\mathrm d})
        \right]_+
        \frac{1+\zeta_k}{\zeta_k}D_f(x_{k-1},x_k)
        \\
        &\quad\le
        \frac12D_f(x_{k-1},x_k)
        \le
        \paren*{1-\frac{\beta_k}{2\bar\beta_k}}D_f(x_{k-1},x_k).
        \label{eq:stepsizerulebound}
    \end{align}

   Combining \eqref{eq:lyapintermediate} and \eqref{eq:stepsizerulebound}, and using $\gamma_{k+1}=\beta_{k+1}\gamma_k$, yields
    \begin{align}
        \Uf
        &\le
        \U-\ga\paren*{1+\rhh-\rf\rhf}\Pb.
    \end{align}
   Finally, $\beta_{k+1}\le\bar\beta_{k+1}$ and $\bar\beta_{k+1}^2=1+\beta_k$ imply
    \begin{align}
        1+\bar\beta_k-\beta_{k+1}\bar\beta_{k+1}
        &\ge
        1+\bar\beta_k-\bar\beta_{k+1}^2
        =\bar\beta_k-\beta_k
        \ge0.
    \end{align}

\end{proof}
\subsection{Consequences of the Lyapunov inequality}
We establish the technical lemmas needed for the convergence guarantees. For $K\ge0$, write
\begin{equation}\label{eq:def-Gmin}
   G_K^{\min}:=\min_{0\le k\le K}G_k.
\end{equation}

First, \cref{lem:lyap} implies the following.
\begin{lemma}\label{lem:lyapunov}
   Under the update rule in \cref{def:adadpgdupdate} and \cref{assump:main}, the following hold.
    \begin{enumerate}
        \item $\paren*{\U}_{k\ge1}$ is nonincreasing and bounded below, hence convergent.
        \item For every $K\ge1$, $\Pmk\le\frac{\UU_1}{\sum^{K+1}_{k=1}\ga}$.
    \end{enumerate}
\end{lemma}
\begin{proof}
We adapt the argument of \citet[Lemma~4.1]{adabregprox2025} to DPGD.
   Since $G_k\ge0$, we have $\U\ge0$, and the first claim follows from \cref{lem:lyap}. In particular,
    \begin{equation}
        0\le\Uf\le\U-\ga(1+\rhh-\rhf\rf)\Pb.
    \end{equation}
Summing from $k=1$ to $K$ gives
\begin{align}
    &\Pmk \sum_{k=1}^{K} \gamma_k \paren*{ 1+\bar{\beta}_k -\beta_{k+1}\bar{\beta}_{k+1} }\le \sum_{k=1}^{K} \gamma_k \paren*{ 1+\bar{\beta}_k -\beta_{k+1}\bar{\beta}_{k+1} } G_{k-1} \le \UU_1-\UU_{K+1}. \label{eq:lyap-telescoping}
\end{align}
The definition of $\UU_{K+1}$ and nonnegativity of its terms give
\begin{align}
    &\UU_{K+1} \ge \gamma_{K+1} \paren*{1+\bar{\beta}_{K+1}}G_K \ge \gamma_{K+1} \paren*{1+\bar{\beta}_{K+1}}\Pmk. \label{eq:terminal-lower-bound}
\end{align}
Combining \eqref{eq:lyap-telescoping} and
\eqref{eq:terminal-lower-bound} yields
\begin{align}
    \UU_1
    &\ge
    \Pmk
    \left[
        \sum_{k=1}^{K}
        \gamma_k
        \paren*{
            1+\bar{\beta}_k
            -\beta_{k+1}\bar{\beta}_{k+1}
        }
        +
        \gamma_{K+1}
        \paren*{1+\bar{\beta}_{K+1}}
    \right].
    \label{eq:pmin-pre}
\end{align}
Since $\gamma_{k+1}=\beta_{k+1}\gamma_k$,
\begin{align}
    &\sum_{k=1}^{K} \gamma_k \paren*{ 1+\bar{\beta}_k -\beta_{k+1}\bar{\beta}_{k+1} } + \gamma_{K+1} \paren*{1+\bar{\beta}_{K+1}} \\
    &\quad = \sum_{k=1}^{K} \paren*{ \gamma_k +\gamma_k\bar{\beta}_k -\gamma_{k+1}\bar{\beta}_{k+1} } + \gamma_{K+1} +\gamma_{K+1}\bar{\beta}_{K+1} = \gamma_1\bar{\beta}_1 + \sum_{k=1}^{K+1}\gamma_k \ge \sum_{k=1}^{K+1}\gamma_k. \label{eq:coefficient-telescoping}
\end{align}
Thus \eqref{eq:pmin-pre} gives
\begin{align}
    \Pmk
    \le
    \frac{\UU_1}{
        \gamma_1\bar{\beta}_1
        +\sum_{k=1}^{K+1}\gamma_k
    }\le
    \frac{\UU_1}{
        \sum_{k=1}^{K+1}\gamma_k
    },
\end{align}
proving the second claim.
\end{proof}
This result reduces the convergence argument to showing $\sum^\infty_{k=1}\ga=\infty$.
\begin{comment}
   For unconstrained DPGD, we also have the following auxiliary result.
\begin{lemma}\label{lem:shiftedseq}
   Suppose that \cref{assump:main} holds. Let $\setE{y_k}\subset\RR^d$ satisfy $y_k\to\xo$ and, for every $k\in\NN$, set
    \begin{equation}
        \bar{y}_k:=y_k-\gf\np{y_k},
    \end{equation}
    where $\setE{\ga}$ is a bounded positive sequence. Then $\bar{y}_k\to\xo$, $\nf(\bar{y}_k)\to0$, and $\gf(\ph{\bar{y}_k}-\min\phi)\to0$.
\end{lemma}
\begin{proof}
   This follows directly from \cref{assump:main} and \eqref{eq:update}.
\end{proof}
\end{comment}

\subsection{Local estimates for vanishing stepsizes}
The next lemma gives the limit of the curvature ratio for the forward map along subsequences whose stepsizes approach zero.
\begin{lemma}\label{lem:gammatozero}
Under \cref{assump:main}, suppose that the sequence $\setE{\x}$ generated by \cref{def:adadpgdupdate} lies in a compact set $\mathcal C\subset\RR^d$. Then, for any index sequence $\setE{k_j}$,
\begin{equation}\label{eq:convex-subsequence-gamma-zero}
    \gamma_{k_j}\to0
\end{equation}
implies
\begin{equation}\label{eq:convex-subsequence-R-one}
    \mathcal R_{k_j,2\bar\beta_{k_j+1}}\to1.
\end{equation}

\end{lemma}
\begin{proof}
We adapt the argument of \citet[Lemma~4.2]{adabregprox2025} to DPGD.
   Set $u_k:=\x-\xb$ and $v_k:=\np{\x}-\np{\xb}$. Then $T_k(\x)-T_k(\xb)=u_k-\ga v_k$.
   By \cref{lem:locallip}, $\nabla\phi\circ\nabla f$ is Lipschitz continuous and bounded on $\mathcal C$. Hence there exist $L_F,M_F>0$ such that
    \begin{equation}\label{eq:convex-u-v-bounds}
        \norm{v_k}\le L_F\norm{u_k},
        \qquad
        \norm{u_k}=\gamma_k\norm{\np{x_{k-1}}}\le M_F\gamma_k.
    \end{equation}
    By \eqref{eq:convex-subsequence-gamma-zero}, we may assume $\gamma_{k_j}L_F\le1$ for sufficiently large $j$. Thus \eqref{eq:convex-u-v-bounds} gives
    \begin{equation}\label{eq:convex-subsequence-T-bound}
        \norm{T_{k_j}(x_{k_j})-T_{k_j}(x_{k_j-1})}
        \le
        \norm{u_{k_j}}+\gamma_{k_j}\norm{v_{k_j}}
        \le
        2M_F\gamma_{k_j}
        \to0.
    \end{equation}
   Moreover, \eqref{ineq:beta-bound} gives $2\bar\beta_{k_j+1}\le2\rhm$, and therefore
    \begin{equation}
        2\bar\beta_{k_j+1}
        \norm{T_{k_j}(x_{k_j})-T_{k_j}(x_{k_j-1})}
        \to0.
    \end{equation}
   Thus we can choose $\xi_j$ on the segment between $x_{k_j}$ and $x_{k_j}+2\bar\beta_{k_j+1}(T_{k_j}(x_{k_j})-T_{k_j}(x_{k_j-1}))$, and $\widetilde\xi_j$ on the segment between $x_{k_j}$ and $x_{k_j-1}$, such that
    \begin{align}
        \mathcal R_{k_j,2\bar\beta_{k_j+1}}
        &=
        \frac{
            \inpr{\nabla^2f(\xi_j)(u_{k_j}-\gamma_{k_j}v_{k_j})}
            {u_{k_j}-\gamma_{k_j}v_{k_j}}
        }{
            \inpr{\nabla^2f(\widetilde\xi_j)u_{k_j}}{u_{k_j}}
        }
        \\
        &=
        \frac{\inpr{\nabla^2f(\xi_j)u_{k_j}}{u_{k_j}}}
        {\inpr{\nabla^2f(\widetilde\xi_j)u_{k_j}}{u_{k_j}}}
        -2\gamma_{k_j}
        \frac{\inpr{\nabla^2f(\xi_j)u_{k_j}}{v_{k_j}}}
        {\inpr{\nabla^2f(\widetilde\xi_j)u_{k_j}}{u_{k_j}}}
        +\gamma_{k_j}^2
        \frac{\inpr{\nabla^2f(\xi_j)v_{k_j}}{v_{k_j}}}
        {\inpr{\nabla^2f(\widetilde\xi_j)u_{k_j}}{u_{k_j}}}.
        \label{eq:ratiohess}
    \end{align}
Equations~\eqref{eq:convex-u-v-bounds} and \eqref{eq:convex-subsequence-T-bound} give $u_{k_j}\to0$ and
    \begin{equation}
        \norm{\xi_j-\widetilde\xi_j}
        \le
        2\bar\beta_{k_j+1}
        \norm{T_{k_j}(x_{k_j})-T_{k_j}(x_{k_j-1})}
        +\norm{u_{k_j}}
        \to0.
    \end{equation}
   Hence a compact convex set $\mathcal K$ contains $\mathcal C,\setE{\xi_j},\setE{\widetilde\xi_j}$.
   Apply \cref{lem:local-strong-convexity} with $h=f$ and denote the Hessian bounds on $\mathcal K$ by $m^f_{\mathcal K},M^f_{\mathcal K}>0$. Equations~\eqref{eq:ratiohess} and \eqref{eq:convex-u-v-bounds} give
    \begin{align}
        \abs*{\mathcal R_{k_j,2\bar\beta_{k_j+1}}-1}
        &\le
        \frac{1}{m^f_{\mathcal K}}
        \norm{\nabla^2f(\xi_j)-\nabla^2f(\widetilde\xi_j)}
        +2\gamma_{k_j}\frac{M^f_{\mathcal K}L_F}{m^f_{\mathcal K}}
        +\gamma_{k_j}^2\frac{M^f_{\mathcal K}L_F^2}{m^f_{\mathcal K}}.
        \label{eq:convex-R-ratio-bound}
    \end{align}
   Since $\nabla^2f$ is uniformly continuous on $\mathcal K$ and $\norm{\xi_j-\widetilde\xi_j}\to0$, the right-hand side of \eqref{eq:convex-R-ratio-bound} tends to $0$. This proves \eqref{eq:convex-subsequence-R-one}.
\end{proof}

\subsection{Iterate boundedness and nonsummability of stepsizes}
\begin{proposition}\label{prop:claims}
Under the update rule in \cref{def:adadpgdupdate} and \cref{assump:main}, the following hold.
\begin{enumerate}
    \item There is a compact set $\K$ containing $\setE{\x}\subset\RR^d$. Consequently, a compact set $\K^*\subset\Omega_\phi$ also contains $\setE{\nf(\x)}$.\label[proposition]{claim:one}
    \item
    \begin{equation}\label{eq:convex-sum-gamma}
        \sum_{k=1}^{\infty}\gamma_k=\infty.
    \end{equation}
\label[proposition]{claim:two}
    \item A subsequence of $\setE{\x}$ converges to $\xo$. Correspondingly, a subsequence of $\setE{\nf(\x)}$ converges to $0$.\label[proposition]{claim:three}
\end{enumerate}
\end{proposition}
\begin{proof}
    (i) By \cref{lem:lyapunov}, for all $k\ge1$,
    \begin{equation}
      D_f(\x,\xo)\le\U\le\UU_1.
    \end{equation}
Hence
    \begin{equation}
       x_k\in\setE{x\in\RR^d|f(x)\le f(\xo)+\UU_1},
        \qquad k\ge1.
    \end{equation}
   The right-hand side is compact because $f$ is level bounded, and adjoining $x_0$ preserves compactness. This proves the claim; the statement for the dual sequence follows by continuity of $\nf$.

    (ii) Suppose, toward a contradiction, that \eqref{eq:convex-sum-gamma} fails. Then $\sum_k\gamma_k<\infty$, so $\gamma_k\to0$. Choose a compact convex set $\mathcal K$ containing the compact set in \cref{claim:one}. By \eqref{eq:local-bregman-ratio-bound} in \cref{lem:local-strong-convexity}, $\zeta_k\ge\bar\zeta$ for some $\bar\zeta>0$. Applying local dual relative smoothness in both orders to $x_k,x_{k-1}$ gives $\widehat L_k^{\mathrm d}\le\overline L^{\mathrm d}$ for some $\overline L^{\mathrm d}>0$.
   Applying \cref{lem:gammatozero} with $k_j=j$ yields
    \begin{equation}
        \mathcal R_{k,2\rhf}\to1.
    \end{equation}
The second candidate in \eqref{eq:stepsize-rule} therefore satisfies
    \begin{equation}
        \frac{\zeta_k}{1 + \zeta_k}\frac{1}{2\rhf[\mathcal R_{k,2\rhf}-(1-\ga \widehat L_k^{\mathrm d})]_+}
        \geq
        \frac{\bar\zeta}{1 + \bar\zeta}
        \frac{1}{2\rhm[\mathcal R_{k,2\rhf}-(1-\ga\overline L^{\mathrm d})]_+}
        \to \infty,
    \end{equation}
where $[\mathcal R_{k,2\rhf}-(1-\ga\overline L^{\mathrm d})]_+\to0$ follows from $\gamma_k\to0$ and the preceding limit.
   Thus $\beta_{k+1}=\bar\beta_{k+1}$ for sufficiently large $k$. Advancing the index once more gives $\beta_k=\bar\beta_k>1$, so
    \begin{equation}
        \bar\beta_{k+1}=\sqrt{1+\beta_k}\ge\sqrt2
    \end{equation}
   for all sufficiently large $k$. Hence $\gamma_{k+1}=\bar\beta_{k+1}\gamma_k\ge\sqrt2\gamma_k$, forcing $\gamma_k$ to diverge, contrary to $\gamma_k\to0$.

    (iii) \cref{claim:two} and \cref{lem:lyapunov} give $G_K^{\min}\to0$. Choose indices with $G_{k_j}\to0$. By \cref{claim:one}, pass to a further subsequence if necessary so that $x_{k_j}\to\bar x$. Continuity of $\phi\circ\nabla f$ and \cref{assump:main} give $\bar x=\xo$, and hence $\nabla f(x_{k_j})\to0$.
\end{proof}

In view of \cref{claim:three}, it remains to prove $\x\to\xo$.
We now use these auxiliary results to give complete proofs of the main results.

\subsection{Convergence of the entire sequence}\label[appendix]{app:proof-convex-convergence}
\begin{proof}[Proof of \cref{thm:main}]
We adapt the argument of \citet[Theorem~4.4]{adabregprox2025} to DPGD.
   By \cref{lem:lyapunov}, $\setE{\U}$ is nonnegative and nonincreasing.
   Thus there exists $\UU\ge0$ such that
    \begin{equation}\label{eq:limit-lyapunov}
        \U\to\UU.
    \end{equation}
We show that $\UU=0$.

   By \cref{claim:three}, there exists a subsequence $\setE{x_{k_j}}$ such that
    \begin{equation}\label{eq:optimal-subsequence}
       x_{k_j}\to\xo.
    \end{equation}
We distinguish cases according to the corresponding stepsizes $\setE{\gamma_{k_j+1}}$.

   First suppose that $\setE{\gamma_{k_j+1}}$ has a bounded subsequence.
   Pass to that subsequence and retain the notation.
   Since $\np{\xo}=0$, continuity of $\np{\cdot}$ and the update give
    \begin{align}
        \norm{x_{k_j+1}-\xo}
        \le
        \norm{x_{k_j}-\xo}
        +
        \gamma_{k_j+1}\norm{\np{x_{k_j}}}\to0.
    \end{align}
   By \eqref{eq:optimal-subsequence} and continuity of $\phi\circ\nf$,
    \begin{equation}
       G_{k_j}\to0.
    \end{equation}
    Since $\setE{\bar{\beta}_{k_j+1}}$ is also bounded,
    \begin{equation}
        \gamma_{k_j+1}
        \paren*{1+\bar{\beta}_{k_j+1}}
       G_{k_j}
        \to0.
    \end{equation}
    Furthermore,
    \begin{equation}
       D_f(x_{k_j+1},\xo)\to0,
        \qquad
       D_f(x_{k_j},x_{k_j+1})\to0,
    \end{equation}
    so the definition of the Lyapunov function gives
    \begin{equation}
        \mathcal{E}_{k_j+1}\to0.
    \end{equation}
    Equation~\eqref{eq:limit-lyapunov} therefore implies $\UU=0$.

   Next suppose that $\setE{\gamma_{k_j+1}}$ has no bounded subsequence.
   Then
    \begin{equation}\label{eq:diverging-next-stepsize}
        \gamma_{k_j+1}\to\infty.
    \end{equation}
    The upper bound on the stepsize ratio,
    \begin{equation}
        \gamma_{k+1}
        \le
        \rhm\gamma_k,
    \end{equation}
    together with \eqref{eq:diverging-next-stepsize}, gives
    \begin{equation}\label{eq:diverging-current-stepsize}
        \gamma_{k_j}
        \ge
        \frac{\gamma_{k_j+1}}{\rhm}
        \to\infty.
    \end{equation}

   Nonnegativity and boundedness of the Lyapunov function give
    \begin{equation}
        0
        \le
        \gamma_{k_j}G_{k_j-1}
        \le
        \mathcal{E}_{k_j}
        \le
        \mathcal{E}_1.
    \end{equation}
    Hence \eqref{eq:diverging-current-stepsize} yields
    \begin{equation}\label{eq:previous-gap-zero}
       G_{k_j-1}\to0.
    \end{equation}

   By \cref{claim:one}, $\setE{x_{k_j-1}}$ is bounded.
   For any cluster point $\bar{x}$, \eqref{eq:previous-gap-zero} and continuity of $\phi\circ\nf$ give
    \begin{equation}
        \ph{\bar{x}}=\phi(0).
    \end{equation}
    By \cref{assump:main}, $\bar{x}=\xo$. Thus
    \begin{equation}\label{eq:previous-iterate-convergence}
       x_{k_j-1}\to\xo.
    \end{equation}

   Convexity of $\phi$ gives, for any $y\in\Omega_\phi$,
    \begin{equation}
        \phi(y)-\phi(0)
        \le
        \inpr{\nabla\phi(y)}{y}.
    \end{equation}
    Applying this to $y=\nf(x_{k_j-1})$ and using the DPGD update yields
    \begin{align}
    &\gamma_{k_j}G_{k_j-1} \le \gamma_{k_j} \inpr{\np{x_{k_j-1}}}{\nf(x_{k_j-1})} = \inpr{ x_{k_j-1}-x_{k_j} }{ \nf(x_{k_j-1}) } \to0. \label{eq:weighted-gap-zero}
\end{align}
   The last limit uses \eqref{eq:optimal-subsequence}, \eqref{eq:previous-iterate-convergence}, and $\nf(\xo)=0$.

   Since $\setE{\bar{\beta}_k}$ is bounded and
    \begin{equation}
        0
        \le
        1-\frac{\beta_k}{2\bar{\beta}_k}
        \le1,
    \end{equation}
    \eqref{eq:optimal-subsequence}, \eqref{eq:previous-iterate-convergence}, \eqref{eq:weighted-gap-zero}, and the definition of the Lyapunov function give
    \begin{align}
        \mathcal{E}_{k_j}
        ={}&
       D_f(x_{k_j},\xo)
        +
        \gamma_{k_j}
        \paren*{1+\bar{\beta}_{k_j}}
       G_{k_j-1}+
        \paren*{
            1-\frac{\beta_{k_j}}{2\bar{\beta}_{k_j}}
        }
       D_f(x_{k_j-1},x_{k_j})
        \to0.
    \end{align}
   Thus $\UU=0$ in this case as well.

   In either case,
    \begin{equation}
        \U\to0.
    \end{equation}
    In particular,
    \begin{equation}
        0
        \le
       f(\x)-f(\xo)
        =
       D_f(\x,\xo)
        \le
        \U
        \to0.
    \end{equation}
    Boundedness from \cref{claim:one}, continuity of $f$, and uniqueness of $\xo$ imply
    \begin{equation}
        \x\to\xo.
    \end{equation}
    Continuity of $\nf$ then gives $\nf(\x)\to0$.

\end{proof}

\subsection{Linear convergence}\label[appendix]{app:proof-best-iterate-rate}

\begin{proof}[Proof of \cref{thm:bestiteraterate}]
If a stationary point is reached, it is $x_\star$ by \cref{assump:main}; extending the iterates constantly gives \eqref{eq:linear-last-iterate} after enlarging the constants to cover the finite prefix. We therefore consider an infinite nonstationary sequence.

By \cref{claim:one}, choose a compact convex set $\mathcal C\subset\RR^d$ containing $x_\star$ and every iterate. The set $\nabla f(\mathcal C)$ is a compact subset of the open convex set $\Omega_\phi$, so choose a compact convex set $\mathcal K^*\subset\Omega_\phi$ containing it. In particular, $0=\nabla f(x_\star)\in\mathcal K^*$. Write $\mu_\phi:=\mu_{\mathcal K^*}>0$ from \cref{assump:rate}. By \cref{lem:local-strong-convexity,lem:locallip}, there are $m_f,M_f,L_F>0$ such that

\begin{equation}\label{eq:linear-local-bounds}
    m_fI\preceq\nabla^2f(x)\preceq M_fI,
    \qquad
    \norm{\np{x}-\np{y}}\le L_F\norm{x-y},
    \qquad x,y\in\mathcal C,
\end{equation}

With $\beta_{\max}=(1+\sqrt5)/2$, initialization and \eqref{ineq:beta-bound} give

\begin{equation}\label{eq:linear-ratio-bounds}
    0<\beta_k\le\bar\beta_k\le\beta_{\max},
    \qquad \bar\beta_k>1,
    \qquad \bar\beta_{k+1}^{\,2}=1+\beta_k,
    \qquad k\ge1.
\end{equation}

Strong convexity of $\phi$ on $\mathcal K^*$ and the local Hessian bounds yield

\begin{align}
    &\Delta_\phi(\nabla f(x_k),\nabla f(x_{k-1})) \ge\mu_\phi\norm{\nabla f(x_k)-\nabla f(x_{k-1})}^2, \\
    &\Delta_f(x_k,x_{k-1}) \le\norm{\nabla f(x_k)-\nabla f(x_{k-1})}\norm{x_k-x_{k-1}} \le m_f^{-1}\norm{\nabla f(x_k)-\nabla f(x_{k-1})}^2.
\end{align}

Applying local dual relative smoothness in both orders and the quadratic Bregman bounds on $\mathcal C$ therefore gives a constant $\overline L^{\mathrm d}>0$ such that

\begin{equation}\label{eq:linear-curvature-bounds}
    0<\underline L^{\mathrm d}:=\mu_\phi m_f
    \le\widehat L_k^{\mathrm d}\le\overline L^{\mathrm d},
    \qquad
    \zeta_k\ge\underline\zeta:=m_f/M_f>0.
\end{equation}

We first bound the stepsizes. If $\gamma_k\ge2/\underline L^{\mathrm d}$, nonnegativity of $\mathcal R_{k,2\bar\beta_{k+1}}$ gives

\[
    \mathcal R_{k,2\bar\beta_{k+1}}-(1-\gamma_k\widehat L_k^{\mathrm d})
    \ge\gamma_k\underline L^{\mathrm d}-1
    \ge\tfrac12\gamma_k\underline L^{\mathrm d}.
\]

Thus \eqref{eq:stepsize-rule} implies $\gamma_{k+1}\le1/\underline L^{\mathrm d}$. Otherwise, $\gamma_{k+1}\le\beta_{\max}\gamma_k<2\beta_{\max}/\underline L^{\mathrm d}$. Since $\gamma_1=\gamma_0$,

\begin{equation}\label{eq:linear-gamma-upper}
    \gamma_k\le\overline\gamma
    :=\max\{\gamma_0,2\beta_{\max}/\underline L^{\mathrm d}\},
    \qquad k\ge0.
\end{equation}

Because $\nabla\phi\circ\nabla f$ is bounded on $\mathcal C$, the points $z_k:=x_k-2\bar\beta_{k+1}\gamma_k\np{x_k}$ are bounded. Choose a compact convex set $\mathcal Q$ containing $\mathcal C$ and all $z_k$, and let $M_{\mathcal Q}^f>0$ bound the largest eigenvalue of $\nabla^2f$ on $\mathcal Q$. The update and \eqref{eq:linear-local-bounds} imply

\[
    \norm{\np{x_k}}
    \le\norm{\np{x_{k-1}}}+L_F\norm{x_k-x_{k-1}}
    \le(1+L_F\overline\gamma)\norm{\np{x_{k-1}}}.
\]

Since $T_k(x_{k-1})=x_k$, the point $z_k$ is the evaluation point in \eqref{eq:local-estimates} for $\delta=2\bar\beta_{k+1}$. Using $x_k-x_{k-1}=-\gamma_k\np{x_{k-1}}$ and $\np{x_{k-1}}\ne0$, the quadratic bounds give

\begin{align}
    &\mathcal R_{k,2\bar\beta_{k+1}} =\frac{2D_f(z_k,x_k)}{4\bar\beta_{k+1}^{\,2}\Delta_f(x_k,x_{k-1})} \le\frac{M_{\mathcal Q}^f}{m_f} \frac{\norm{\np{x_k}}^2}{\norm{\np{x_{k-1}}}^2} \le\frac{M_{\mathcal Q}^f}{m_f}(1+L_F\overline\gamma)^2 =:\overline{\mathcal R}. \label{eq:linear-R-upper}
\end{align}

Combining \eqref{eq:linear-curvature-bounds}--\eqref{eq:linear-R-upper} with the stepsize rule, including its division convention, yields

\begin{equation}\label{eq:linear-beta-lower}
    \beta_{k+1}\ge\underline\beta
    :=\min\left\{1,
      \frac{\underline\zeta}{1+\underline\zeta}
      \frac{1}{2\beta_{\max}(\overline{\mathcal R}
                   +\overline\gamma\,\overline L^{\mathrm d})}\right\}>0.
\end{equation}

Moreover, there exists $\gamma_{\rm safe}>0$ such that

\begin{equation}\label{eq:linear-small-step}
    0<\gamma_k\le\gamma_{\rm safe}
    \quad\Longrightarrow\quad
    \beta_{k+1}=\bar\beta_{k+1}.
\end{equation}

Otherwise, there are indices $k_j$ with $\gamma_{k_j}\to0$ and $\beta_{k_j+1}<\bar\beta_{k_j+1}$. By \cref{lem:gammatozero} and \eqref{eq:linear-curvature-bounds},

\[
    [\mathcal R_{k_j,2\bar\beta_{k_j+1}}
       -(1-\gamma_{k_j}\widehat L_{k_j}^{\mathrm d})]_+\to0.
\]

Since $\zeta_{k_j}\ge\underline\zeta>0$ and $\bar\beta_{k_j+1}\le\beta_{\max}$, the second candidate in \eqref{eq:stepsize-rule} tends to $+\infty$, contradicting the strict inequality. Set $\underline\gamma:=\min\{\gamma_1,\underline\beta\gamma_{\rm safe}\}>0$. If $\gamma_k\le\gamma_{\rm safe}$, \eqref{eq:linear-small-step} gives $\gamma_{k+1}\ge\gamma_k$; otherwise, \eqref{eq:linear-beta-lower} gives $\gamma_{k+1}\ge\underline\beta\gamma_{\rm safe}$. Induction and \eqref{eq:linear-gamma-upper} therefore yield

\begin{equation}\label{eq:linear-stepsize-bounds}
    0<\underline\gamma\le\gamma_k\le\overline\gamma<\infty,
    \qquad k\ge1.
\end{equation}

Next we compare the Lyapunov energy with $G_{k-1}$. Since $0$ minimizes $\phi$ and belongs to $\Omega_\phi$, $\nabla\phi(0)=0$ and $\np{x_\star}=0$. Strong convexity on $\mathcal K^*$ and \eqref{eq:linear-local-bounds} give

\begin{equation}\label{eq:linear-G-lower}
    G_j\ge\frac{\mu_\phi}{2}\norm{\nabla f(x_j)}^2
         \ge\frac{\mu_\phi m_f^2}{2}\norm{x_j-x_\star}^2,
    \qquad j\ge0.
\end{equation}

The update gives

\[
    \norm{x_k-x_\star}\le(1+\overline\gamma L_F)\norm{x_{k-1}-x_\star},
    \qquad
    \norm{x_k-x_{k-1}}\le\overline\gamma L_F\norm{x_{k-1}-x_\star}.
\]

Using $D_f(x_k,x_\star)=f(x_k)-f_\star$, $1-\beta_k/(2\bar\beta_k)\in[1/2,1]$, and the upper quadratic bounds in \eqref{eq:convex-lyapunov}, we obtain

\begin{align}
    0\le f(x_k)-f_\star\le\mathcal E_k
    &\le\frac{M_f}{2}\left[(1+\overline\gamma L_F)^2
                          +\overline\gamma^{\,2}L_F^2\right]\norm{x_{k-1}-x_\star}^2\notag\\
    &\quad+\overline\gamma(1+\beta_{\max})G_{k-1}
     \le C_EG_{k-1},
    \label{eq:linear-energy-comparison}
\end{align}

where the finite constant

\[
    C_E:=\overline\gamma(1+\beta_{\max})
      +\frac{M_f}{\mu_\phi m_f^2}
         \left[(1+\overline\gamma L_F)^2
                        +\overline\gamma^{\,2}L_F^2\right]>0
\]

is independent of $k$.

Finally, set $c_k:=1+\bar\beta_k-\beta_{k+1}\bar\beta_{k+1}$. By \eqref{eq:linear-ratio-bounds}, $c_k\ge\bar\beta_k-\beta_k\ge0$, and \cref{lem:lyap} gives

\begin{equation}\label{eq:linear-descent}
    \mathcal E_{k+1}\le\mathcal E_k-\gamma_kc_kG_{k-1}.
\end{equation}

For any $n\ge1$ and integer $N\ge1$, the identity $\gamma_k\beta_{k+1}=\gamma_{k+1}$ gives

\begin{align}
    &\sum_{k=n}^{n+N-1}\gamma_kc_k =\sum_{k=n}^{n+N-1}\gamma_k +\gamma_n\bar\beta_n-\gamma_{n+N}\bar\beta_{n+N} \ge N\underline\gamma-\beta_{\max}\overline\gamma. \label{eq:linear-block-weight}
\end{align}

Fix $N$ large enough that $N\underline\gamma>\beta_{\max}\overline\gamma$. Summing \eqref{eq:linear-descent} and using \eqref{eq:linear-energy-comparison}, $\gamma_kc_k\ge0$, and monotonicity of the energy yields

\begin{align}
    &\mathcal E_n-\mathcal E_{n+N} \ge\sum_{k=n}^{n+N-1}\gamma_kc_kG_{k-1} \ge C_E^{-1}\sum_{k=n}^{n+N-1}\gamma_kc_k\mathcal E_k \ge\frac{N\underline\gamma-\beta_{\max}\overline\gamma}{C_E}\mathcal E_{n+N}.
\end{align}

Thus, with $\kappa:=\bigl(1+(N\underline\gamma-\beta_{\max}\overline\gamma)/C_E\bigr)^{-1}\in(0,1)$ independent of $n$,

\begin{equation}\label{eq:linear-energy-geometric}
    \mathcal E_{n+N}\le \kappa\mathcal E_n,
    \qquad
    \mathcal E_k\le\mathcal E_1\kappa^{\lfloor(k-1)/N\rfloor},
    \qquad n,k\ge1.
\end{equation}

Put $\alpha:=\kappa^{1/(2N)}\in(0,1)$. Since $f(x_k)-f_\star\le\mathcal E_k$ and $\mathcal E_{k+1}\ge\underline\gamma G_k$, \eqref{eq:linear-energy-geometric} and \eqref{eq:linear-G-lower} imply

\[
    f(x_k)-f_\star\le\frac{\mathcal E_1}{\kappa\alpha^2}\alpha^{2k}
      \quad(k\ge1),
    \qquad
    \norm{\nabla f(x_k)}^2
      \le\frac{2\mathcal E_1}{\kappa\underline\gamma\mu_\phi}\alpha^{2k}
      \quad(k\ge0).
\]

Enlarging $C_f$ to cover $k=0$ proves \eqref{eq:linear-last-iterate}, since $\alpha^{2k}\le \alpha^k$.
\end{proof}

\section{Proofs for the Weakly Convex Analysis}\label[appendix]{app:weak-convergence-proofs}
Throughout this appendix, $\eta>0$ and the algorithm uses \eqref{eq:stepsize-rule}.

\subsection{Proof of the weakly convex Lyapunov inequality}\label[appendix]{app:proof-weak-lyapunov}
Induction using the weakly convex stepsize rule gives
\begin{equation}\label{eq:weak-betabar-order}
    0<\beta_k\le\bar\beta_k\le\beta_{\mathrm{cap}}.
\end{equation}

The definitions of $m_k^\eta$ and $\zeta_k^\eta$ give
\begin{equation}\label{eq:weak-eta-absorb}
    \eta\norm{x_k-x_{k-1}}^2
    =
    \frac{\eta}{m_k^\eta}\Delta_{g_\eta}(x_k,x_{k-1}).
\end{equation}
They also give
\begin{equation}\label{eq:weak-delta-zeta}
    \Delta_{g_\eta}(x_k,x_{k-1})
    =
    \frac{1+\zeta_k^\eta}{\zeta_k^\eta}D_{g_\eta}(x_{k-1},x_k).
\end{equation}

\begin{proof}[Proof of \cref{lem:weak-lyapunov}]
   Multiply \eqref{eq:common-gap-identity} in \cref{lem:basic-two-step-identities} by $\gamma_{k+1}/2$ and
   apply \eqref{eq:common-distance-identity} with $\bar x=x_\star$.
   Using \eqref{eq:weak-curvature-rewrite} and $D_\phi\ge0$ gives
    \begin{align}
    &D_f(x_{k+1},x_\star) +\frac32\gamma_{k+1}G_k +D_f(x_k,x_{k+1}) \\
    &\quad \le D_f(x_k,x_\star) +\frac12\gamma_{k+1}G_{k-1} -\frac{\beta_{k+1}}2 \left( 1-\gamma_k \widehat L_k^{\mathrm d,\eta}-\frac{\eta}{m_k^\eta} \right) \Delta_{g_\eta}(x_k,x_{k-1}) +\chi_{k+1}. \label{eq:weak-combined-two-step}
\end{align}

   By the update and the definition of $\chi_{k+1}$,
    \begin{align}
        \chi_{k+1}
        &=
        \inpr{\nabla g_\eta(x_{k+1})-\nabla g_\eta(x_k)}{x_{k+1}-x_k}
        -\eta\norm{x_{k+1}-x_k}^2.
    \end{align}
   As in \eqref{eq:convex-cross-term-bound}, convexity of $g_\eta$ and \eqref{eq:bregman-three-point} give
    \begin{align}
    &\chi_{k+1} \le \frac{\beta_{k+1}}{\bar\beta_{k+1}}D_{g_\eta}(x_k,x_{k+1}) +\frac{\bar\beta_{k+1}\beta_{k+1}}2 \mathcal R^\eta_{k,\bar\beta_{k+1}}\Delta_{g_\eta}(x_k,x_{k-1}) -\eta\norm{x_{k+1}-x_k}^2. \label{eq:weak-cross-term-bound}
\end{align}
   Substituting \eqref{eq:weak-cross-term-bound} into \eqref{eq:weak-combined-two-step} and using
    \begin{equation}
       D_f(x_k,x_{k+1})
        =D_{g_\eta}(x_k,x_{k+1})-\frac{\eta}{2}\norm{x_{k+1}-x_k}^2
    \end{equation}
   yields
    \begin{align}
        &D_f(x_{k+1},x_\star)
        +\frac32\gamma_{k+1}G_k
        +\left(
            1-\frac{\beta_{k+1}}{\bar\beta_{k+1}}
        \right)D_{g_\eta}(x_k,x_{k+1})
        +\frac{\eta}{2}\norm{x_{k+1}-x_k}^2
        \\
        &\quad\le
       D_f(x_k,x_\star)
        +\frac12\gamma_{k+1}G_{k-1}
        \\
        &\qquad
        +\beta_{k+1}
        \left(
            \frac{\bar\beta_{k+1}}2\mathcal R^\eta_{k,\bar\beta_{k+1}}
            -\frac12
            \left(
                1-\gamma_k \widehat L_k^{\mathrm d,\eta}-\frac{\eta}{m_k^\eta}
            \right)
        \right)
        \Delta_{g_\eta}(x_k,x_{k-1}).
        \label{eq:weak-basic-inequality}
    \end{align}

   By \eqref{eq:weak-betabar-order}, the numerator of the second candidate in \eqref{eq:stepsize-rule} is nonnegative. Thus \eqref{eq:stepsize-rule} gives
    \begin{align}
        &\beta_{k+1}
        \left(
            \frac{\bar\beta_{k+1}}2\mathcal R^\eta_{k,\bar\beta_{k+1}}
            -\frac12
            \left(
                1-\gamma_k \widehat L_k^{\mathrm d,\eta}-\frac{\eta}{m_k^\eta}
            \right)
        \right)
        \\
        &\quad\le
        \beta_{k+1}
        \left[
            \frac{\bar\beta_{k+1}}2\mathcal R^\eta_{k,\bar\beta_{k+1}}
            -\frac12
            \left(
                1-\gamma_k \widehat L_k^{\mathrm d,\eta}-\frac{\eta}{m_k^\eta}
            \right)
        \right]_+
        \\
        &\quad\le
        \frac{\zeta_k^\eta}{1+\zeta_k^\eta}
        \left(
            1-\frac{\beta_k}{\bar\beta_k}
        \right)
        +\frac{\eta}{2m_k^\eta}.
        \label{eq:weak-stepsize-absorption}
    \end{align}
   Multiply \eqref{eq:weak-stepsize-absorption} by $\Delta_{g_\eta}(x_k,x_{k-1})$ and use \eqref{eq:weak-eta-absorb} and \eqref{eq:weak-delta-zeta} to obtain
    \begin{align}
        &\beta_{k+1}
        \left(
            \frac{\bar\beta_{k+1}}2\mathcal R^\eta_{k,\bar\beta_{k+1}}
            -\frac12
            \left(
                1-\gamma_k \widehat L_k^{\mathrm d,\eta}-\frac{\eta}{m_k^\eta}
            \right)
        \right)
        \Delta_{g_\eta}(x_k,x_{k-1})
        \\
        &\quad\le
        \left(
            1-\frac{\beta_k}{\bar\beta_k}
        \right)D_{g_\eta}(x_{k-1},x_k)
        +\frac{\eta}{2}\norm{x_k-x_{k-1}}^2.
        \label{eq:weak-coefficient-control}
    \end{align}
   Substituting \eqref{eq:weak-coefficient-control} into \eqref{eq:weak-basic-inequality} and using $\gamma_{k+1}=\beta_{k+1}\gamma_k$ gives
    \begin{align}
        \mathcal{E}_{k+1}^{\mathrm{weak}}
        &\le
        \mathcal{E}_k^{\mathrm{weak}}
        -\frac{\gamma_k}{2}(3-\beta_{k+1})G_{k-1},
    \end{align}
which is \eqref{eq:weak-lyapunov-descent}.
   Finally, $\beta_{k+1}\le\beta_{\mathrm{cap}}<3$ yields
    \begin{equation}\label{eq:weak-lyapunov-descent-uniform}
        \mathcal{E}_{k+1}^{\mathrm{weak}}
        \le
        \mathcal{E}_k^{\mathrm{weak}}
        -\frac{\gamma_k}{2}(3-\beta_{\mathrm{cap}})G_{k-1}.
    \end{equation}
\end{proof}
\subsection{Consequences of the Lyapunov inequality}

By \eqref{eq:weak-nonnegative-gaps} and \eqref{eq:weak-betabar-order}, every term in \eqref{eq:weak-lyapunov} is nonnegative.
As in the convex analysis, the Lyapunov function first gives the following.
\begin{lemma}\label{lem:weak-lyapunov-consequence}
   Under \cref{assump:main} with $\eta>0$ and \cref{def:weak-update-rule}, the following hold.
    \begin{enumerate}
        \item $\setE{\mathcal{E}_k^{\mathrm{weak}}}$ is nonnegative and nonincreasing, hence convergent.
        \item For every $K\ge1$,
        \begin{equation}\label{eq:weak-best-iterate}
            \min_{0\le j\le K-1}G_j
            \le
            \frac{2}{3-\beta_{\mathrm{cap}}}
            \frac{
                \mathcal{E}_1^{\mathrm{weak}}
            }{
                \sum_{k=1}^{K}\gamma_k
            }.
        \end{equation}
    \end{enumerate}
   Furthermore,
    \begin{equation}\label{eq:weak-weighted-summability}
        \sum_{k=1}^{\infty}\gamma_kG_{k-1}<\infty.
    \end{equation}

\end{lemma}
\begin{proof}
   Equation~\eqref{eq:weak-lyapunov-descent-uniform} implies that $\mathcal{E}_k^{\mathrm{weak}}$ is nonnegative and nonincreasing, and for every $K\ge1$,
    \begin{equation}\label{eq:weak-telescoping}
        \mathcal{E}_{K+1}^{\mathrm{weak}}
        +
        \frac{3-\beta_{\mathrm{cap}}}{2}
        \sum_{k=1}^{K}\gamma_kG_{k-1}
        \le
        \mathcal{E}_1^{\mathrm{weak}}.
    \end{equation}
This proves the first claim; letting $K\to\infty$ gives \eqref{eq:weak-weighted-summability}. Moreover, \eqref{eq:weak-telescoping} and $G_{k-1}\ge\min_{0\le j\le K-1}G_j$ yield
    \begin{equation}
        \frac{3-\beta_{\mathrm{cap}}}{2}
        \paren*{\min_{0\le j\le K-1}G_j}
        \sum_{k=1}^{K}\gamma_k
        \le
        \mathcal{E}_1^{\mathrm{weak}},
    \end{equation}
proving \eqref{eq:weak-best-iterate}.
\end{proof}

\subsection{Local estimates for vanishing stepsizes}
For convex objectives, \cref{lem:gammatozero} analyzes the limit of the local estimate along subsequences whose stepsizes tend to $0$. The contradiction argument in \cref{claim:two} then proves $\sum_k\gamma_k=\infty$.
For weakly convex objectives, we do not rule out stepsizes approaching $0$.
Instead, we show that even if $\gamma_k\to0$, the decrease is slow enough to preserve $\sum_k\gamma_k=\infty$.
Besides local boundedness of $m_k^\eta,\widehat L_k^{\mathrm d,\eta}$, the next lemma quantifies how $\mathcal R^\eta_{k,\bar\beta_{k+1}}$ approaches $1$.

\begin{lemma}[local estimates]\label{lem:weak-local-estimates}
   Under \cref{assump:main} with $\eta>0$, suppose that $\setE{x_k}$ lies in a compact set and $\gamma_k\to0$.
   Then there exist constants $0<\underline m\le\overline m<\infty$, $\overline L^{\mathrm d}<\infty$, $C_{\mathcal R}>0$, and $k_0\in\NN$ such that, for all $k\ge k_0$,
    \begin{equation}\label{eq:weak-m-l-bounds}
        \underline m
        \le
       m_k^\eta
        \le
        \overline m,
        \qquad
        0\le \widehat L_k^{\mathrm d,\eta}\le\overline L^{\mathrm d}
    \end{equation}
   and
    \begin{equation}\label{eq:weak-R-rate}
        \abs*{\mathcal R^\eta_{k,\bar\beta_{k+1}}-1}
        \le
       C_{\mathcal R}\gamma_k.
    \end{equation}

\end{lemma}
\begin{proof}
   Let $\mathcal C$ be the compact convex hull of a compact set containing the iterates.
   Applying \cref{lem:local-strong-convexity} to $h=g_\eta$ yields constants $\underline m,\overline m>0$ such that,
   by \eqref{eq:local-symmetric-gradient-bounds},
    \begin{equation}\label{eq:weak-m-exact-bound}
       \underline m
        \le
       m_k^\eta
        =
        \frac{\Delta_{g_\eta}(x_k,x_{k-1})}{\norm{x_k-x_{k-1}}^2}
        \le
       \overline m.
    \end{equation}
Applying \eqref{eq:local-comparison-sym} on $\mathcal C$ gives $\overline L^{\mathrm d}>0$ such that
    \begin{equation}\label{eq:weak-l-exact-bound}
        0\le \widehat L_k^{\mathrm d,\eta}\le \overline L^{\mathrm d}.
    \end{equation}

   It remains to prove \eqref{eq:weak-R-rate}. Set
    \begin{equation}
       u_k:=x_k-x_{k-1},
        \qquad
       v_k:=\np{x_k}-\np{x_{k-1}}.
    \end{equation}
By \cref{prop:local-regularity}(ii), $\nabla\phi\circ\nabla f$ is Lipschitz continuous and bounded on $\mathcal C$.
   Let $L_F>0$ be its Lipschitz constant and $M_F>0$ a bound on its norm on $\mathcal C$. The update gives
    \begin{equation}\label{eq:weak-preconditioned-local-bounds}
        \norm{v_k}\le L_F\norm{u_k},
        \qquad
        \norm{u_k}
        =\gamma_k\norm{\np{x_{k-1}}}
        \le M_F\gamma_k.
    \end{equation}
   The definition of $\bar\beta_{k+1}$ in the weakly convex setting also gives
    \begin{equation}\label{eq:weak-betabar-exact-bound}
        1\le\bar\beta_{k+1}\le\beta_{\mathrm{cap}},
        \qquad
        0\le\bar\beta_{k+1}-1\le s\gamma_k.
    \end{equation}
   Since $\gamma_k\to0$, we may assume $\gamma_kL_F\le1$ for sufficiently large $k$. Hence
    \begin{equation}\label{eq:weak-T-bound}
        \norm{T_k(x_k)-T_k(x_{k-1})}
        =\norm{u_k-\gamma_kv_k}
        \le2M_F\gamma_k.
    \end{equation}

   Setting $z_k:=x_k+\bar\beta_{k+1}(T_k(x_k)-T_k(x_{k-1}))$, \eqref{eq:weak-betabar-exact-bound} and \eqref{eq:weak-T-bound} give
    \begin{equation}
        \norm{z_k-x_k}\le2\beta_{\mathrm{cap}} M_F\gamma_k.
    \end{equation}
   Thus there is a compact convex set $\mathcal K$ containing $\mathcal C$ and all $z_k$ for sufficiently large $k$.
   Since $g_\eta$ is $C^2$ with a locally Lipschitz continuous Hessian, \cref{lem:local-strong-convexity} gives $m_g,M_g,L_g>0$ such that, for all $x,y\in\mathcal K$,
    \begin{equation}\label{eq:weak-local-regularity-bounds}
       m_gI\preceq\nabla^2g_\eta(x)\preceq M_gI,
        \qquad
        \norm{\nabla^2g_\eta(x)-\nabla^2g_\eta(y)}
        \le L_g\norm{x-y}.
    \end{equation}

   By Taylor's theorem, choose $\xi_k$ on the segment joining $x_k,z_k$ and $\widetilde\xi_k$ on the segment joining $x_{k-1},x_k$ such that
    \begin{equation}
        \mathcal R^\eta_{k,\bar\beta_{k+1}}
        =
        \frac{
            \inpr{\nabla^2g_\eta(\xi_k)(u_k-\gamma_kv_k)}{u_k-\gamma_kv_k}
        }{
            \inpr{\nabla^2g_\eta(\widetilde\xi_k)u_k}{u_k}
        }.
    \end{equation}
Expanding the Hessian quotient for $g_\eta$ as after \eqref{eq:ratiohess}, and using \eqref{eq:weak-preconditioned-local-bounds} and \eqref{eq:weak-local-regularity-bounds}, gives
    \begin{align}
        \abs*{\mathcal R^\eta_{k,\bar\beta_{k+1}}-1}
        &\le
        \frac{1}{m_g}
        \norm{\nabla^2g_\eta(\xi_k)-\nabla^2g_\eta(\widetilde\xi_k)}
        +2\gamma_k\frac{M_gL_F}{m_g}
        +\gamma_k^2\frac{M_gL_F^2}{m_g}.
        \label{eq:weak-R-ratio-bound}
    \end{align}
   The locations of $\xi_k,\widetilde\xi_k$ and \eqref{eq:weak-preconditioned-local-bounds}, \eqref{eq:weak-betabar-exact-bound}, and \eqref{eq:weak-T-bound} imply
    \begin{align}
    &\norm{\xi_k-\widetilde\xi_k} \le \bar\beta_{k+1}\norm{T_k(x_k)-T_k(x_{k-1})} +\norm{u_k} \le (2\beta_{\mathrm{cap}}+1)M_F\gamma_k. \label{eq:weak-xi-eta-bound}
\end{align}
   Hence \eqref{eq:weak-local-regularity-bounds} yields
    \begin{equation}
        \norm{\nabla^2g_\eta(\xi_k)-\nabla^2g_\eta(\widetilde\xi_k)}
        \le
       L_g(2\beta_{\mathrm{cap}}+1)M_F\gamma_k.
    \end{equation}
   Increase $k_0$ so that $\gamma_k\le1$. Equation~\eqref{eq:weak-R-ratio-bound} then gives
    \begin{equation}
        \abs*{\mathcal R^\eta_{k,\bar\beta_{k+1}}-1}
        \le C_{\mathcal R}\gamma_k,
    \end{equation}
   with
    \begin{equation}\label{eq:weak-C-R}
       C_{\mathcal R}
        :=
        \frac{
           L_g(2\beta_{\mathrm{cap}}+1)M_F
            +2M_gL_F
            +M_gL_F^2
        }{m_g}.
    \end{equation}
This proves the result.
\end{proof}

\subsection{Iterate boundedness and nonsummability of stepsizes}
We next prove the counterpart of \cref{prop:claims} needed for the weakly convex analysis.
\begin{proposition}\label{prop:weak-convergence-preparation}
   Under \cref{assump:main} with $\eta>0$ and \cref{def:weak-update-rule}, the following hold.
    \begin{enumerate}
        \item
       There exists a compact set $\mathcal C\subset\RR^d$ containing $\setE{x_k}$.
        \label[proposition]{claim:weak-one}
        \item\label[proposition]{claim:weak-two}
        \begin{equation}\label{eq:weak-sum-gamma}
            \sum_{k=1}^{\infty}\gamma_k=\infty.
        \end{equation}

        \item\label[proposition]{claim:weak-three}
        \begin{equation}\label{eq:weak-G-liminf}
            \liminf_{k\to\infty}G_k=0.
        \end{equation}

    \end{enumerate}
\end{proposition}
\begin{proof}
    (i)
   Use the sublevel-set argument from \cref{claim:one}. By \cref{lem:weak-lyapunov-consequence},
    \begin{equation}
       f(x_k)-f(x_\star)
        =D_f(x_k,x_\star)
        \le\mathcal{E}_k^{\mathrm{weak}}
        \le\mathcal{E}_1^{\mathrm{weak}},
        \qquad k\ge1.
    \end{equation}
   Hence $\setE{x_k}_{k\ge1}$ lies in a single sublevel set of $f$. This set is compact by level boundedness in \cref{assump:main}(i), and adjoining $x_0$ preserves compactness, proving (i).

    (ii) If $\gamma_k\not\to0$, there exist $\varepsilon>0$ and infinitely many indices with $\gamma_k\ge\varepsilon$, so \eqref{eq:weak-sum-gamma} follows immediately.
   Suppose therefore that $\gamma_k\to0$.
   By (i) and \cref{lem:weak-local-estimates}, there exist $\overline m,\overline L^{\mathrm d},C_{\mathcal R}>0$ and $k_0$ such that, for all $k\ge k_0$,
    \begin{equation}
       m_k^\eta\le\overline m,
        \qquad
        \widehat L_k^{\mathrm d,\eta}\le\overline L^{\mathrm d},
        \qquad
        \abs*{\mathcal R^\eta_{k,\bar\beta_{k+1}}-1}
        \le C_{\mathcal R}\gamma_k.
    \end{equation}
   Equation~\eqref{eq:weak-betabar-exact-bound} also gives
    \begin{align}
    &\abs*{ \bar\beta_{k+1}\mathcal R^\eta_{k,\bar\beta_{k+1}}-1 } \le \bar\beta_{k+1} \abs*{\mathcal R^\eta_{k,\bar\beta_{k+1}}-1} + \abs*{\bar\beta_{k+1}-1} \le (\beta_{\mathrm{cap}} C_{\mathcal R}+s)\gamma_k. \label{eq:weak-betabar-R-bound}
\end{align}

   Denote half the expression inside $[\cdot]_+$ in the denominator of the second candidate in \eqref{eq:stepsize-rule} by
    \begin{equation}\label{eq:weak-A-k}
       A_k
        :=
        \frac{\bar\beta_{k+1}}2\mathcal R^\eta_{k,\bar\beta_{k+1}}
        -
        \frac12\left(
            1-\gamma_k\widehat L_k^{\mathrm d,\eta}-\frac{\eta}{m_k^\eta}
        \right).
    \end{equation}
By the definition of $A_k$,
    \begin{equation}
       A_k-\frac{\eta}{2m_k^\eta}
        =
        \frac12
        \left(
            \bar\beta_{k+1}\mathcal R^\eta_{k,\bar\beta_{k+1}}-1
        \right)
        +
        \frac12\gamma_k\widehat L_k^{\mathrm d,\eta}.
    \end{equation}
   Thus \eqref{eq:weak-betabar-R-bound} yields
    \begin{equation}\label{eq:weak-A-a-bound}
        \abs*{A_k-\frac{\eta}{2m_k^\eta}}
        \le \frac{\eta}{2\overline m}C_0\gamma_k
        \le \frac{\eta}{2m_k^\eta}C_0\gamma_k,
        \qquad
       C_0:=\frac{\overline m}{\eta}\bigl(\beta_{\mathrm{cap}} C_{\mathcal R}+s+\overline L^{\mathrm d}\bigr).
    \end{equation}
   Since $\gamma_k\to0$, increase $k_0$ so that $C_0\gamma_k\le1/2$. Then
    \begin{equation}\label{eq:weak-A-positive}
       0<\frac{\eta}{2m_k^\eta}(1-C_0\gamma_k)
       \le A_k\le\frac{\eta}{2m_k^\eta}(1+C_0\gamma_k).
    \end{equation}

   Using $1-\beta_k/\bar\beta_k\ge0$ from \eqref{eq:weak-betabar-order}, the second candidate in \eqref{eq:stepsize-rule} satisfies
    \begin{align}
    &\frac{ \displaystyle \frac{\zeta_k^\eta}{1+\zeta_k^\eta} \left( 1-\frac{\beta_k}{\bar\beta_k} \right) +\frac{\eta}{2m_k^\eta} }{ A_k }
    \ge \frac{1}{1+C_0\gamma_k}\ge1-C_0\gamma_k. \label{eq:weak-second-candidate-lower}
\end{align}
   The last inequality uses $1/(1+t)\ge1-t$ for $t\ge0$.
   Since \eqref{eq:weak-betabar-exact-bound} gives $\bar\beta_{k+1}\ge1$, \eqref{eq:stepsize-rule} and \eqref{eq:weak-second-candidate-lower} imply
    \begin{equation}\label{eq:weak-beta-lower-local}
        \beta_{k+1}\ge1-C_0\gamma_k.
    \end{equation}

   Since $C_0\gamma_k\le1/2$, we obtain
    \begin{equation}\label{eq:weak-gamma-recursion}
        \gamma_{k+1}
        \ge
        \gamma_k(1-C_0\gamma_k)>0.
    \end{equation}
   Hence
    \begin{align}
    &\frac1{\gamma_{k+1}} \le \frac1{\gamma_k(1-C_0\gamma_k)} \le \frac1{\gamma_k}+2C_0, \label{eq:weak-reciprocal-recursion}
\end{align}
   where the last inequality uses $(1-t)^{-1}\le1+2t$ for $0\le t\le1/2$.
   Summing \eqref{eq:weak-reciprocal-recursion} from $k_0$ to $k-1$ gives
    \begin{equation}\label{eq:weak-harmonic-lower}
        \gamma_k
        \ge
        \frac{1}{
            \gamma_{k_0}^{-1}
            +2C_0(k-k_0)
        },
        \qquad k\ge k_0.
    \end{equation}
   The series on the right diverges as a harmonic series, proving \eqref{eq:weak-sum-gamma}.

    (iii)
   Use \eqref{eq:weak-weighted-summability} from
    \cref{lem:weak-lyapunov-consequence} and
    \cref{claim:weak-two}.
   If $\liminf_{k\to\infty}G_k>0$, there exist $\varepsilon>0$ and $k_1$ such that $G_{k-1}\ge\varepsilon$ for all $k\ge k_1$.
   Therefore,
    \begin{equation}
        \sum_{k=k_1}^{\infty}\gamma_kG_{k-1}
        \ge
        \varepsilon\sum_{k=k_1}^{\infty}\gamma_k
        =
        \infty
    \end{equation}
   contradicting \eqref{eq:weak-weighted-summability}.
   This proves \eqref{eq:weak-G-liminf}.
\end{proof}

Finally, we prove stationarity along the entire sequence.
For convex objectives, convergence of the Lyapunov function gave $x_k\to x_\star$. For weakly convex objectives, convergence to a global minimizer cannot generally be expected.
We instead establish $\phi(\nabla f(x_k))\to\phi(0)$ and hence $\nabla f(x_k)\to0$.

\subsection{Convergence to stationarity}\label[appendix]{app:proof-weak-stationarity}
\begin{proof}[Proof of \cref{thm:weak-stationarity}]
   Let $\mathcal C$ be the compact set in \cref{claim:weak-one}.
   By \cref{assump:main} with $\eta>0$, $\phi\circ\nabla f\in C^1(\RR^d)$,
    so it is Lipschitz continuous on $\operatorname{conv}(\mathcal C)$.
   Since $\np{\cdot}$ is bounded on $\mathcal C$, there exist $L_P,M_{\mathcal C}>0$ such that
    \begin{equation}
        \abs*{
            \phi(\nabla f(x))-\phi(\nabla f(y))
        }
        \le L_P\norm{x-y},
        \qquad
        \norm{\np{x}}\le M_{\mathcal C},
        \qquad
       x,y\in\mathcal C.
    \end{equation}

   The update in
    \cref{def:weak-update-rule} gives
    \begin{align}
    &\abs*{G_{k+1}-G_k} \le L_P\norm{x_{k+1}-x_k} = L_P\gamma_{k+1}\norm{\np{x_k}} \le L_PM_{\mathcal C}\gamma_{k+1}. \label{eq:weak-G-increment}
\end{align}

   By \cref{claim:weak-three}, $\liminf_{k\to\infty}G_k=0$.
   Suppose that $G_k\not\to0$.
   Then, for some $\varepsilon>0$, we can choose infinitely many pairwise disjoint intervals $[a_j,b_j]$ such that $G_{a_j}\ge\varepsilon$ and $b_j>a_j$ is the first subsequent index with $G_{b_j}\le\varepsilon/2$.
   For $a_j\le k<b_j$, $G_k>\varepsilon/2$, and \eqref{eq:weak-G-increment} gives
    \begin{align}
    &\frac\varepsilon2 \le G_{a_j}-G_{b_j} \le \sum_{k=a_j}^{b_j-1} \abs*{G_{k+1}-G_k} \le L_PM_{\mathcal C} \sum_{k=a_j}^{b_j-1}\gamma_{k+1}.
\end{align}
   Hence
    \begin{align}
    &\sum_{k=a_j+1}^{b_j} \gamma_kG_{k-1} \ge \frac\varepsilon2 \sum_{k=a_j+1}^{b_j}\gamma_k \ge \frac{\varepsilon^2}{4L_PM_{\mathcal C}}.
\end{align}
   Each disjoint interval contributes a fixed positive amount to the left-hand side of \eqref{eq:weak-weighted-summability}, contradicting \eqref{eq:weak-weighted-summability}.
   Thus $G_k\to0$.

   Finally, \cref{claim:weak-one} and continuity of $\nabla f$ imply that $\setE{\nabla f(x_k)}$ lies in a compact set.
   For any convergent subsequence $\nabla f(x_{k_j})\to y$, $G_k\to0$ and continuity of $\phi$ give
    \begin{equation}
        \phi(y)=\phi(0).
    \end{equation}

   By \cref{assump:main}(ii), $y=0$.
   Every cluster point of $\setE{\nabla f(x_k)}$ is therefore $0$, proving \eqref{eq:weak-gradient-stationarity}.
\end{proof}

\section{Preconditioner Design: Properties, Proofs, and Examples}\label[appendix]{app:preconditioner-design-details}
We verify the properties of the construction in \cref{sec:preconditioner-design}, derive the exponents for the principal target models, and collect the resulting preconditioners and their relation to existing methods in \cref{tab:preconditioner-recipes-full}.

\subsection{\texorpdfstring{Proof of \cref{prop:power-interpolation-preconditioner}}{Proof of the interpolation preconditioner proposition}}\label[appendix]{app:preconditioner-proofs}
\begin{proof}[Proof of \cref{prop:power-interpolation-preconditioner}]
Fix $a,b$ and suppress these parameters in the notation. Differentiating \eqref{eq:normal-construction}, with $\psi=\psi_{a,b}$ defined in \eqref{eq:power-interpolation}, gives, for $\varrho>0$,
\begin{align}
   \nu_{\mathrm{N}}''(\varrho)
    &=\psi(\varrho),
    &
    \frac{\nu_{\mathrm{N}}'(\varrho)}{\varrho}
    &=\frac{1}{\varrho}\int_0^\varrho\psi(s)\,ds,
    \label{eq:normal-curvatures}
    \\
    \frac{\nu_{\mathrm{T}}'(\varrho)}{\varrho}
    &=\psi(\varrho),
    &
   \nu_{\mathrm{T}}''(\varrho)
    &=\psi(\varrho)
    \frac{1+(a+1)\varrho^b}{1+\varrho^b}.
    \label{eq:tangential-curvatures}
\end{align}
For either construction allowed in the proposition, write $\nu=\nu_{\mathrm{N}}$ or $\nu=\nu_{\mathrm{T}}$ and $\phi(y)=\nu(\norm{y})$. Since $\psi>0$, $\psi(0)=1$, and $a\ge-1$ in the tangential case, both $\nu''(\varrho)$ and $\nu'(\varrho)/\varrho$ are positive and tend to $1$ as $\varrho\to0$. Also $\nu(0)=\nu'(0)=0$, so $\nu(\varrho)=\varrho^2/2+o(\varrho^2)$ and $\nabla\phi(0)=0$. For $y\ne0$ and $\varrho=\norm{y}$,
\begin{equation}\label{eq:normal-hessian}
    \nabla\phi(y)=\frac{\nu'(\varrho)}{\varrho}y,
    \qquad
    \nabla^2\phi(y)
    =\nu''(\varrho)\frac{yy^\top}{\varrho^2}
    +\frac{\nu'(\varrho)}{\varrho}
    \paren*{I-\frac{yy^\top}{\varrho^2}}.
\end{equation}
Thus $\nabla\phi(y)=y+o(\norm{y})$, proving $\nabla^2\phi(0)=I$, and the Hessian formula shows that $\nabla^2\phi(y)\to I$ as $y\to0$. Consequently, $\phi\in C^2(\RR^d)$ and $\nabla^2\phi\succ0$ everywhere. Since $\nu'(\varrho)>0$ for $\varrho>0$, its unique minimizer is $0$. Moreover, $\phi$ is finite and convex on $\RR^d$, so $\Omega_\phi=\RR^d$ contains $\nabla f(\RR^d)$ and $\phi\in\mathcal P$. Its $C^2$ regularity gives local Lipschitz continuity of $\nabla\phi$, proving (i). Combined with the assumed properties of $f$, this also proves (iii).

For (ii), suppose that $\eta=0$ and \cref{assump:main}(i) holds, and fix a nonempty compact set $\mathcal K\subset\RR^d$. Choose compact convex sets $\mathcal C\supset\mathcal K$ and $\mathcal K^*\supset\nabla f(\mathcal C)$. Applying \cref{lem:local-strong-convexity} to $f$ on $\mathcal C$ and $\phi$ on $\mathcal K^*$ gives constants $m_f,M_f,M_\phi>0$ such that, for all $x,y\in\mathcal K$,
\begin{align}
    &D_\phi(\nabla f(x),\nabla f(y))
    \le \frac{M_\phi}{2}\norm{\nabla f(x)-\nabla f(y)}^2
    \le \frac{M_\phi M_f^2}{2}\norm{x-y}^2
    \le \frac{M_\phi M_f^2}{m_f}D_f(y,x).
    \label{eq:designed-preconditioner-local-drs}
\end{align}
This proves \cref{assump:main}(iii), and hence \cref{assump:main}. The same lemma applied to $\phi$ on every nonempty compact convex subset of $\Omega_\phi=\RR^d$ proves \cref{assump:rate}.
\end{proof}

\subsection{\texorpdfstring{Target models and the exponent $a$}{Target models and the exponent a}}\label[appendix]{subsubsec:principal-targets}
For the power-law target $\tau(r)=r^p/p$ with $p>1$, set $q=p/(p-1)$. Since $\varrho=\tau'(r)=r^{p-1}$, both inverse curvatures are proportional to $\varrho^{q-2}$, giving $a=q-2=(2-p)/(p-1)$. These targets cover $a>-1$, with growth exponent $p=(a+2)/(a+1)$.

For the exponential target $\tau(r)=e^r$, we have $\varrho=e^r$ and
\begin{equation}\label{eq:exponential-curvatures}
    h_{\mathrm{N}}(\varrho)=\frac{1}{\varrho},
    \qquad h_{\mathrm{T}}(\varrho)=\frac{\log\varrho}{\varrho}.
\end{equation}
Thus the normal inverse curvature is a power law with $a=-1$. For the tangential construction, we also choose $a=-1$, retaining the factor $\varrho^{-1}$ and omitting $\log\varrho$. The exponent $a$ reflects the target's inverse curvature, while $b>0$ controls the interpolation with curvature $1$ near the origin.

\subsection{Representative preconditioners and existing methods}\label[appendix]{subsubsec:representative-b12}\label[appendix]{app:preconditioner-design-supplement}
\cref{tab:preconditioner-recipes-full} gives the full set of examples, including the additional power-law variants omitted from \cref{tab:preconditioner-recipes}. These illustrate the range of the construction; the empirical comparison of four preconditioners is given in \cref{app:numerical-additional-results}.

\begin{table}[!htb]
\centering
\setlength{\tabcolsep}{5pt}
\renewcommand{\arraystretch}{1.15}
\caption{Preconditioning functions and related methods, where $q=p/(p-1)>1$.}
\label{tab:preconditioner-recipes-full}
\begin{tabularx}{\linewidth}{@{}>{\raggedright\arraybackslash}p{0.55\linewidth}>{\raggedright\arraybackslash\hyphenpenalty=10000}X@{}}
\toprule
Preconditioning function & Relation \\
\midrule
$\displaystyle\phi_{\mathrm{N}}^{0,b}(y)=\phi_{\mathrm{T}}^{0,b}(y)=\frac{\norm{y}^2}{2}\quad(b>0)$ & GD \\
\addlinespace[7pt]
$\displaystyle\phi_{\mathrm{N}}^{q-2,1}(y)=\frac{(1+\norm{y})^q-1-q\norm{y}}{q(q-1)}$ & New variant \\
\addlinespace[7pt]
$\displaystyle\phi_{\mathrm{N}}^{q-2,2}(y)=\int_0^{\norm{y}}(\norm{y}-s)(1+s^2)^{(q-2)/2}\,ds$ & New variant \\
\addlinespace[7pt]
$\displaystyle\phi_{\mathrm{T}}^{q-2,1}(y)=\frac{(1+\norm{y})^q-1}{q}-\frac{(1+\norm{y})^{q-1}-1}{q-1}$ & New variant \\
\addlinespace[7pt]
$\displaystyle\phi_{\mathrm{T}}^{q-2,2}(y)=\frac{(1+\norm{y}^2)^{q/2}-1}{q}$ & $p$-norm DPGD~\citep{maddison2021} \\
\addlinespace[7pt]
$\displaystyle\phi_{\mathrm{N}}^{-1,1}(y)=(1+\norm{y})\log(1+\norm{y})-\norm{y}$ & \citet[Table~1]{oikonomidis2025nonlinearly} \\
\addlinespace[7pt]
$\displaystyle\phi_{\mathrm{N}}^{-1,2}(y)=\norm{y}\operatorname{arsinh}(\norm{y})-\sqrt{1+\norm{y}^2}+1$ & HGD~\citep{oikonomidis2025nonlinearly,oikonomidis2026nonlinearly} \\
\addlinespace[7pt]
$\displaystyle\phi_{\mathrm{T}}^{-1,1}(y)=\phi_{\mathrm{N}}^{-2,1}(y)=\norm{y}-\log(1+\norm{y})$ & NGD; exponential-penalty DPGD~\citep{Zhang2020Why,maddison2021} \\
\addlinespace[7pt]
$\displaystyle\phi_{\mathrm{T}}^{-1,2}(y)=\sqrt{1+\norm{y}^2}-1$ & AdaGrad-type~\citep{adagrad11} \\
\addlinespace[7pt]
$\displaystyle\lim_{b\to\infty}\phi_{\mathrm{T}}^{-1,b}(y)=\begin{cases}\norm{y}^2/2,&\norm{y}\le1,\\\norm{y}-1/2,&\norm{y}>1\end{cases}$ & Gradient clipping~\citep{Zhang2020Why} \\
\bottomrule
\end{tabularx}
\end{table}

For the power-law target $\tau(r)=r^p/p$, the normal construction with $b=2$ retains a one-dimensional integral for general $q$; the other listed finite-$b$ examples have elementary expressions. The equality $\phi_{\mathrm{N}}^{-2,1}=\phi_{\mathrm{T}}^{-1,1}$ also shows that different curvature directions and exponents can yield the same preconditioner.

For the exponential target, the normal constructions additionally recover the tangential logarithmic factor in \eqref{eq:exponential-curvatures}: their tangential curvatures are $\log(1+\varrho)/\varrho$ and $\operatorname{arsinh}(\varrho)/\varrho$, both asymptotic to $\log\varrho/\varrho$. In particular, directly interpolating the target's tangential inverse curvature by $\nu'(\varrho)/\varrho=\log(1+\varrho)/\varrho$ recovers $\phi_{\mathrm{N}}^{-1,1}$. This illustrates interpolation beyond pure power laws without introducing a new preconditioner. By contrast, the two tangential constructions with $a=-1$ have tangential curvature asymptotic to $\varrho^{-1}$, without the logarithmic factor.

In the clipping limit, $\psi_{-1,b}(\varrho)\to1/\max\{1,\varrho\}$, so the update direction is $y/\max\{1,\|y\|\}$. The limiting function is linear in $\varrho$ for $\varrho>1$ and loses positive normal curvature; it is therefore outside \cref{prop:power-interpolation-preconditioner}.

The same function of the radius also defines a separable preconditioner $y\mapsto\sum_{i=1}^d \nu(|y_i|)$. Applying this to the NGD, HGD, and AdaGrad-type entries gives SNGD, SHGD, and memoryless coordinatewise AdaGrad-type updates~\citep{oikonomidis2025nonlinearly,oikonomidis2026nonlinearly,adagrad11}. The AdaGrad-type specialization uses only the current squared gradient, with regularization $1$ inside the square root; the radial entry uses its norm-based counterpart~\citep[Example~1.5 and Lemma~1.3]{oikonomidis2025nonlinearly}. The separable NGD construction also corresponds to Adam with both exponential decay parameters set to zero and denominator regularization $1$~\citep{kingma2017adammethodstochasticoptimization}.

\section{Experimental Details and Additional Results}\label[appendix]{app:numerical-details}
We give the preprocessing, initial points, problem regularity, computational environment, and additional results.

\subsection{Data, preprocessing, and initial points}\label[appendix]{app:dataset-preprocessing}
All experiments use the full distributed datasets; sample counts $n$ and optimization dimensions $d$ are specified for each dataset. For \texttt{kin8nm}, we standardize the features and response without adding an intercept. For \texttt{covtype}, let $A\in\RR^{n\times54}$ denote the feature matrix from \texttt{covtype.libsvm.binary.scale}. We map the labels to $\{-1,1\}$ and whiten the 54 features to 53 dimensions without an intercept: retain eigenvalues of $A^\top A/n$ exceeding $10^{-10}$ times the largest, form $W$ from the corresponding eigenvectors divided by the square roots of their eigenvalues, and set $B=AW$. Regularization and initialization use these optimization coordinates. The distributed features of \texttt{a9a} and \texttt{svmguide1} are used without further standardization or an added intercept. For \texttt{LEDGAR}~\citep{tuggener-etal-2020-ledgar} and \texttt{SCOTUS}, we use the distributed LIBSVM TF--IDF features without further transformation or an intercept. Each distributed training file combines the original LexGLUE~\citep{chalkidis-etal-2022-lexglue} training and validation sets; test data are excluded. We define a binary problem by labeling the most frequent class as $+1$ and all others as $-1$, breaking ties by the smallest numeric label. This gives 3661 positive samples (label 47) for \texttt{LEDGAR} and 1371 (label 0) for \texttt{SCOTUS}. These experiments compare optimization methods, not multiclass prediction accuracy. For \texttt{Skin Non-Skin} and \texttt{Rice Cammeo/Osmancik}, we standardize the features and append an unpenalized intercept; no features were removed. Standardization uses the mean and population standard deviation over all samples.

We initialize $x_0=0$ for \texttt{kin8nm}, \texttt{Skin Non-Skin}, and \texttt{Rice Cammeo/Osmancik}, and $x_0=\nabla f(0)/\norm{\nabla f(0)}$ in the whitened coordinates for \texttt{covtype}. For all powered hinge problems, choose $x_0$ as a positive multiple of $\nabla f(0)$ such that $\max_i[1-y_i a_i^\top x_0]_+=10$.

\subsection{Problem definitions and regularity}\label[appendix]{app:numerical-problem-details}\label[appendix]{app:numerical-powered-hinge-details}\label[appendix]{app:numerical-weak-exp-cauchy-details}
For $p>2$, both $t\mapsto\abs{t}^p$ and $t\mapsto[t]_+^p$ are $C^2$. Thus both convex objectives in \cref{sec:numerical-experiments} are $C^2$ and satisfy $\nabla^2f\succeq I/n$. Together with \cref{prop:power-interpolation-preconditioner}, this gives the convex convergence and rate assumptions.

For the weakly convex objective in \cref{sec:numerical-experiments}, the exponential loss is convex and each Cauchy penalty term satisfies
\begin{equation}
    \frac{\mathrm d^2}{\mathrm dt^2}\left[\omega\log\paren*{1+\frac{t^2}{\sigma^2}}\right]
    =\frac{2\omega(\sigma^2-t^2)}{(\sigma^2+t^2)^2}\ge-\frac{\omega}{4\sigma^2}.
\end{equation}
Thus $f$ is $\eta$-weakly convex with $\eta=\omega/(4\sigma^2)=0.025$. For both datasets, the prepared matrix $A$ has full column rank and both classes are present. The exponential-loss Hessian is positive definite, and the Cauchy penalty plus $\eta\|x\|^2/2$ has a positive semidefinite Hessian, so $\nabla^2g_\eta\succ0$. On each sublevel set, the Cauchy penalty bounds all nonintercept coordinates and the exponential loss from both classes bounds the intercept. Hence $f$ is level bounded; smoothness and \cref{prop:power-interpolation-preconditioner} give the remaining weakly convex assumptions.

\subsection{Reference values, timing, and computational environment}\label[appendix]{app:numerical-environment}
For each convex problem, we select the method requiring the fewest gradient evaluations to reach the common stopping criterion, among the eleven compared methods. We continue the selected method for 20,000 additional iterations and use the lowest objective value obtained as $f_{\rm ref}$. For the runtime comparisons, we use per-iteration medians of 22 runs for convex problems and 10 for weakly convex problems, with balanced execution order; runtimes include initialization but exclude reference-value computation.

After the first step, Armijo trials start at twice the previous accepted stepsize and are halved until acceptance, with no fixed upper cap. The local stepsize estimate in \eqref{eq:local-estimates} requires an auxiliary function evaluation but no gradient at that point; this cost is included in runtimes.

We use an Apple M1 CPU (8 cores), 16\,GB memory, and macOS 26.4.1 (arm64), with double-precision arithmetic and one BLAS/OpenMP thread. The run records specify Python 3.13.14, NumPy 2.5.1, SciPy 1.18.0, and pandas 3.0.5; convex runs also use scikit-learn 1.9.0. For power-law targets, $\phi_{\mathrm N}^{a,2}$ need not have an elementary expression; we evaluate it and its gradient numerically using SciPy's \texttt{hyp2f1} in the $\ell_p$ and powered hinge loss experiments. Other tested preconditioners use elementary formulas.

Numerical safeguards use a $10^{-15}$ floor in Bregman ratios, clip negative Bregman values to zero, and select the growth candidate when curvature denominators are at most $10^{-15}$. Removing these safeguards in 13 tested configurations left trajectories unchanged through the first relative-gradient crossing of $10^{-6}$.

\subsection{Additional numerical results}\label[appendix]{app:numerical-additional-results}
We use $\phi_{\mathrm{N}}^{a,1},\phi_{\mathrm{N}}^{a,2},\phi_{\mathrm{T}}^{a,1},\phi_{\mathrm{T}}^{a,2}$ from \cref{tab:preconditioner-recipes-full}, with $a$ determined by each target. Each figure compares these four preconditioners under Adaptive DPGD and LS, giving eight curves. LS denotes DPGD linesearch~\citep{maddison2021} for convex problems and preconditioned Armijo for weakly convex problems. Colors identify preconditioners, solid lines denote Adaptive DPGD, and dashed lines denote LS; distinct markers help distinguish overlapping curves. Holding the preconditioner fixed compares stepsize rules, while holding the rule fixed compares preconditioners. Convex plots show relative gradient norms and objective gaps against gradient evaluations and runtime; weakly convex plots show relative gradient norms against both horizontal axes. All curves retain the same gradient-based stopping criterion as the main comparison.

\subsubsection{\texorpdfstring{$\ell_p$}{lp} loss}
For \texttt{kin8nm} from OpenML~\citep{kin8nmOpenML} ($n=8192$, $d=8$), we set $p=4$.

\begin{figure}[!htb]
    \centering
    \includegraphics[width=\linewidth]{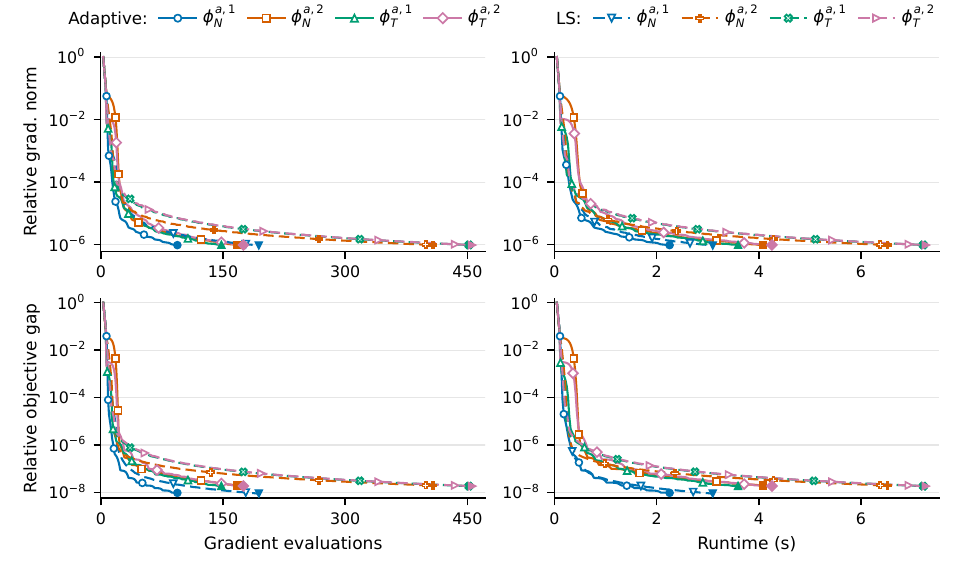}
    \caption{$\ell_6$ loss on \texttt{covtype}. Rows: relative gradient norm and relative objective gap. Columns: gradient evaluations and runtime.}
    \label{fig:detail-covtype}
\end{figure}

\begin{figure}[!htb]
    \centering
    \includegraphics[width=\linewidth]{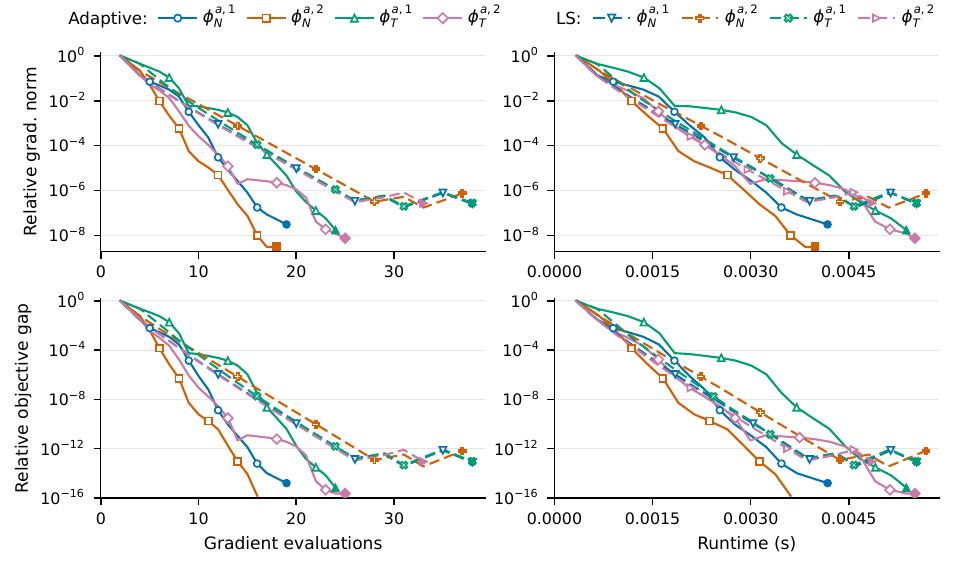}
    \caption{$\ell_4$ loss on \texttt{kin8nm}. Rows and columns follow \cref{fig:detail-covtype}.}
    \label{fig:detail-kin8nm}
\end{figure}

\subsubsection{Powered hinge loss}\label[appendix]{subsubsec:numerical-powered-hinge}
On \texttt{LEDGAR} and \texttt{SCOTUS}, Adaptive DPGD reaches the gradient stopping criterion sooner than LS with the same preconditioner (\cref{fig:detail-ledgar,fig:detail-scotus}). The four Adaptive variants are close in runtime.
We also retain the full comparisons on \texttt{a9a}~\citep{LIBSVM,adult} ($n=32561$, $d=123$) and \texttt{svmguide1}~\citep{LIBSVM} ($n=3089$, $d=4$). Adaptive DPGD reaches the gradient stopping criterion with fewer gradient evaluations and less runtime than LS for each preconditioner (\cref{fig:detail-a9a,fig:detail-svmguide1}).

\begin{figure}[!htb]
    \centering
    \includegraphics[width=\linewidth]{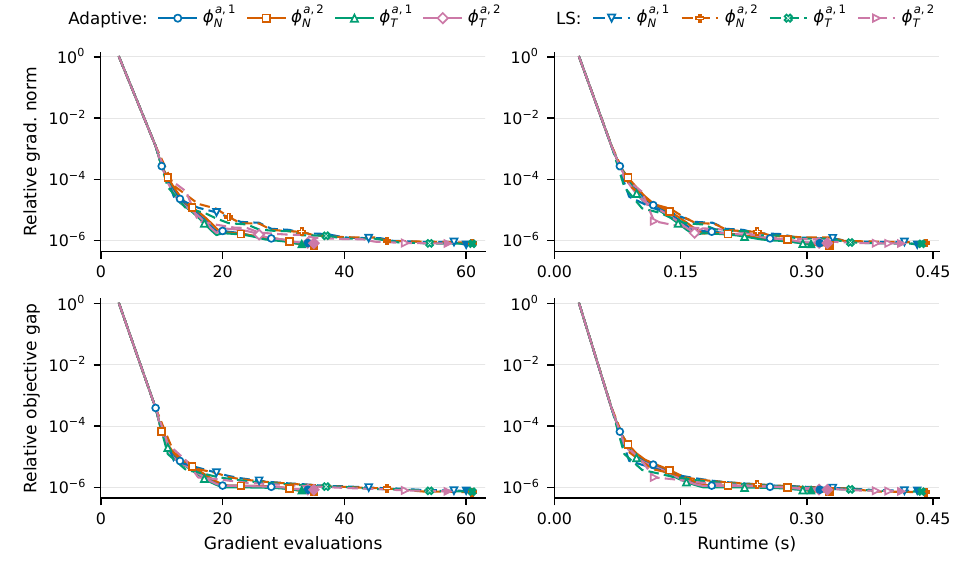}
    \caption{Powered hinge loss ($p=6$) on \texttt{LEDGAR}. Rows and columns follow \cref{fig:detail-covtype}.}
    \label{fig:detail-ledgar}
\end{figure}

\begin{figure}[!htb]
    \centering
    \includegraphics[width=\linewidth]{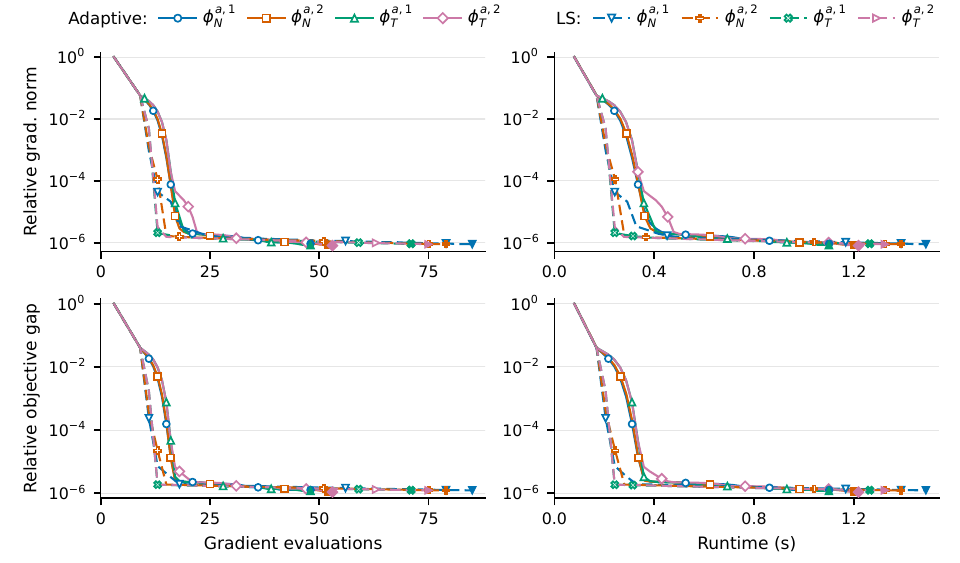}
    \caption{Powered hinge loss ($p=6$) on \texttt{SCOTUS}. Rows and columns follow \cref{fig:detail-covtype}.}
    \label{fig:detail-scotus}
\end{figure}

\begin{figure}[!htb]
    \centering
    \includegraphics[width=\linewidth]{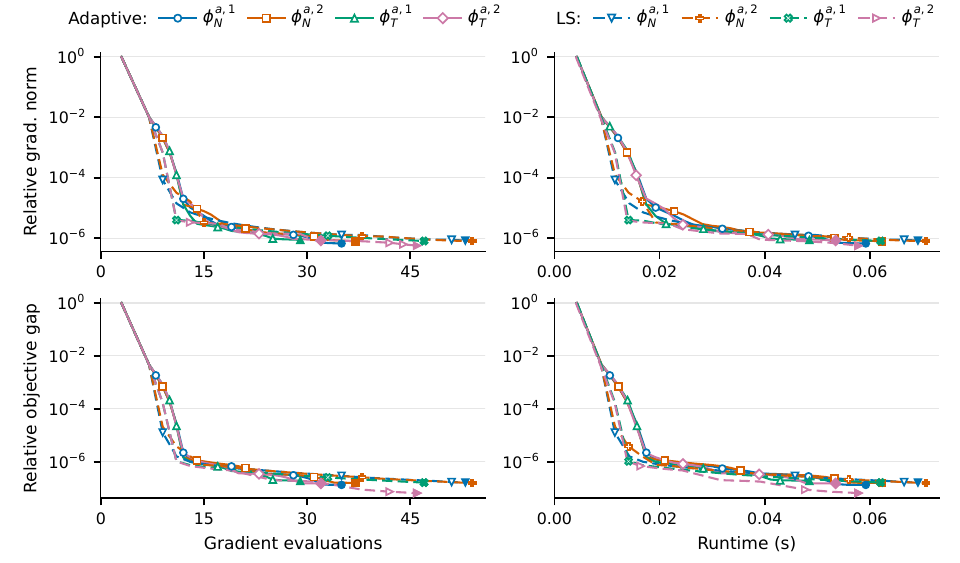}
    \caption{Powered hinge loss ($p=6$) on \texttt{a9a}. Rows and columns follow \cref{fig:detail-covtype}.}
    \label{fig:detail-a9a}
\end{figure}

\begin{figure}[!htb]
    \centering
    \includegraphics[width=\linewidth]{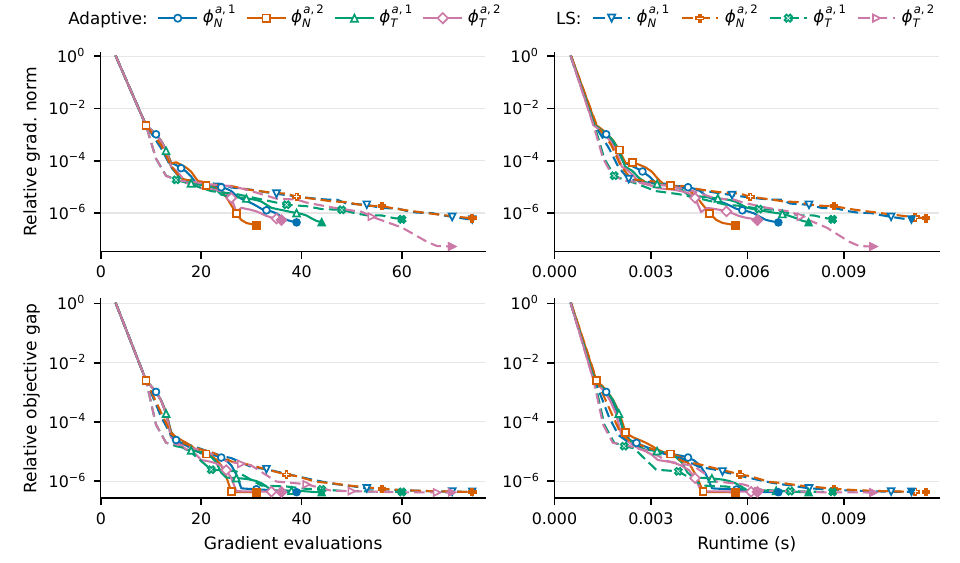}
    \caption{Powered hinge loss ($p=6$) on \texttt{svmguide1}. Rows and columns follow \cref{fig:detail-covtype}.}
    \label{fig:detail-svmguide1}
\end{figure}

\subsubsection{Exponential loss with Cauchy penalty}
On \texttt{Skin Non-Skin}, all Adaptive DPGD variants substantially outperform preconditioned Armijo linesearch (\cref{fig:detail-skin_nonskin}). On \texttt{Rice Cammeo/Osmancik}, Adaptive DPGD variants other than $\phi_{\mathrm{T}}^{-1,1}$ substantially outperform preconditioned Armijo linesearch (\cref{fig:detail-rice_cammeo}).

\begin{figure}[!htb]
    \centering
    \includegraphics[width=\linewidth]{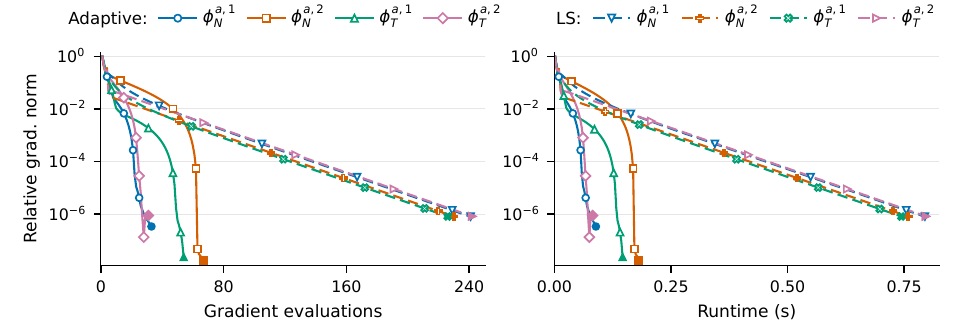}
    \caption{Exponential loss with Cauchy penalty on \texttt{Skin Non-Skin}. Relative gradient norms against gradient evaluations (left) and runtime (right).}
    \label{fig:detail-skin_nonskin}
\end{figure}

\begin{figure}[!htb]
    \centering
    \includegraphics[width=\linewidth]{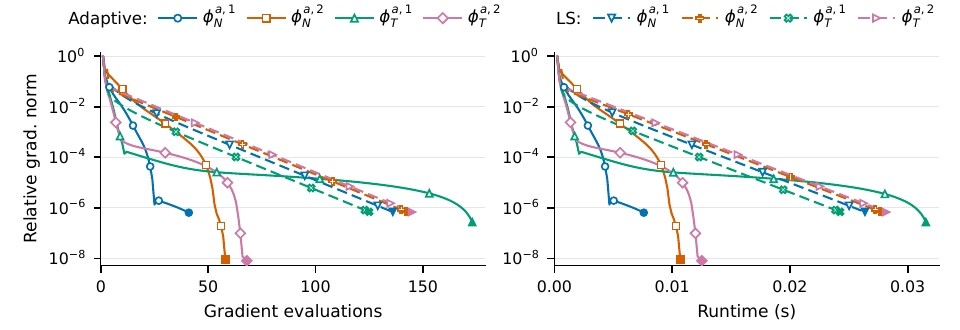}
    \caption{Exponential loss with Cauchy penalty on \texttt{Rice Cammeo/Osmancik}. Relative gradient norms against gradient evaluations (left) and runtime (right).}
    \label{fig:detail-rice_cammeo}
\end{figure}

\clearpage
\subsubsection{Summary of observations}\label[appendix]{app:numerical-observations}
Under Adaptive DPGD, the $p$-norm regression preconditioner of \citet{maddison2021} is $\phi_{\mathrm{T}}^{a,2}$, whereas $\phi_{\mathrm{N}}^{a,2}$ performs better on \texttt{kin8nm} (\cref{fig:detail-kin8nm}) and $\phi_{\mathrm{N}}^{a,1}$ on \texttt{covtype} (\cref{fig:detail-covtype}). Likewise, $\phi_{\mathrm N}^{-1,1}$ outperforms their exponential-penalty preconditioner $\phi_{\mathrm T}^{-1,1}$ on both weakly convex problems (\cref{fig:detail-skin_nonskin,fig:detail-rice_cammeo}). Together with the powered hinge results, these observations illustrate the practical value of the design freedom allowed by our local conditions: alternative preconditioners can be constructed and applied across different losses without establishing a problem-specific global dual relative smoothness bound.

The numerical evaluation of $\phi_{\mathrm N}^{a,2}$ through \texttt{hyp2f1} does not dominate runtime in the convex experiments. With Adaptive DPGD, all evaluations of this preconditioner together account for less than $9\%$ of total runtime in each of the six convex problems (median share over 22 timing repeats). Its gains in gradient evaluations are retained in runtime on \texttt{kin8nm} and \texttt{svmguide1}, where it reaches the stopping criterion fastest among the four tested preconditioners under Adaptive DPGD.

Overall, the normal constructions $\phi_{\mathrm N}^{a,1}$ and $\phi_{\mathrm N}^{a,2}$ provide effective choices among the tested preconditioners, with good performance in both gradient evaluations and runtime across several losses.

The comparison with DPGD linesearch also shows the benefit of avoiding repeated trial evaluations. After initialization, Adaptive DPGD uses one gradient evaluation per update, whereas the linesearch of \citet{maddison2021} evaluates gradients at trial points. On \texttt{covtype} and all four powered hinge datasets, Adaptive DPGD reaches the stopping criterion with fewer gradient evaluations and less runtime for each of the four preconditioners. For example, with $\phi_{\mathrm N}^{a,1}$ on \texttt{covtype}, linesearch takes fewer updates but uses 194 gradient evaluations and 3.10 seconds, compared with 94 and 2.25 seconds for Adaptive DPGD. This illustrates how a linesearch-free update can improve runtime even when more updates are needed.

\clearpage

\subsection{Full comparisons for the main experiments}\label[appendix]{app:main-full-comparisons}

\begin{figure}[!htb]
    \centering
    \includegraphics[width=\linewidth]{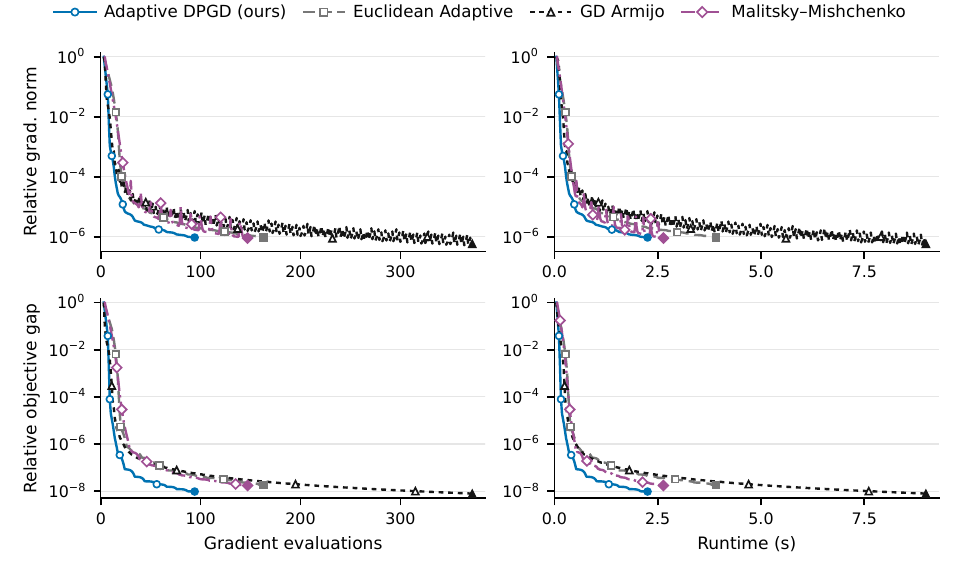}
    \caption{Full comparison for \cref{fig:numerical-main}(a) on \texttt{covtype}. Rows: relative gradient norm and relative objective gap. Columns: gradient evaluations and runtime.}
    \label{fig:main-full-covtype}
\end{figure}

\begin{figure}[!htb]
    \centering
    \includegraphics[width=\linewidth]{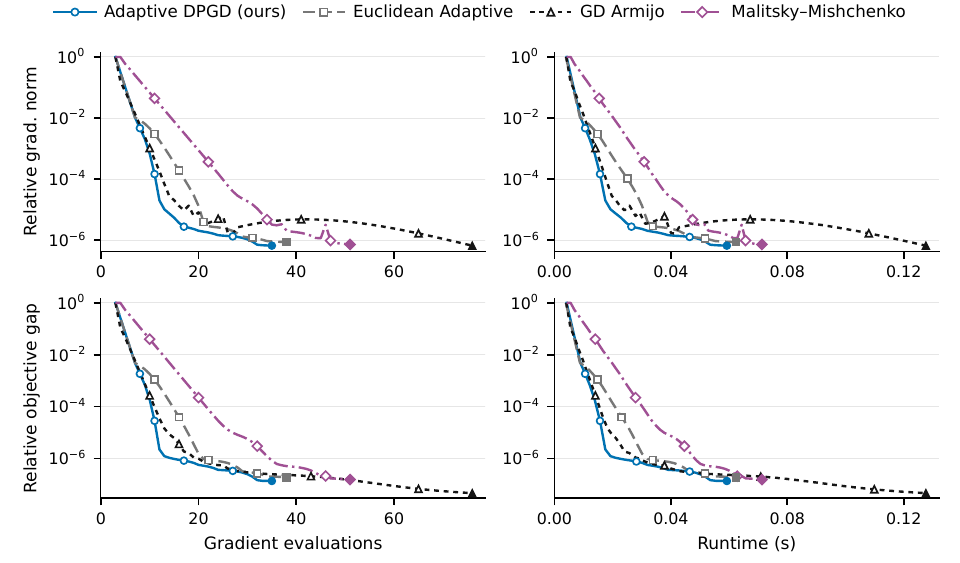}
    \caption{Full comparison for \cref{fig:numerical-main}(b) on \texttt{a9a}. Rows and columns follow \cref{fig:main-full-covtype}.}
    \label{fig:main-full-a9a}
\end{figure}

\clearpage

\begin{figure}[!htb]
    \centering
    \includegraphics[width=\linewidth]{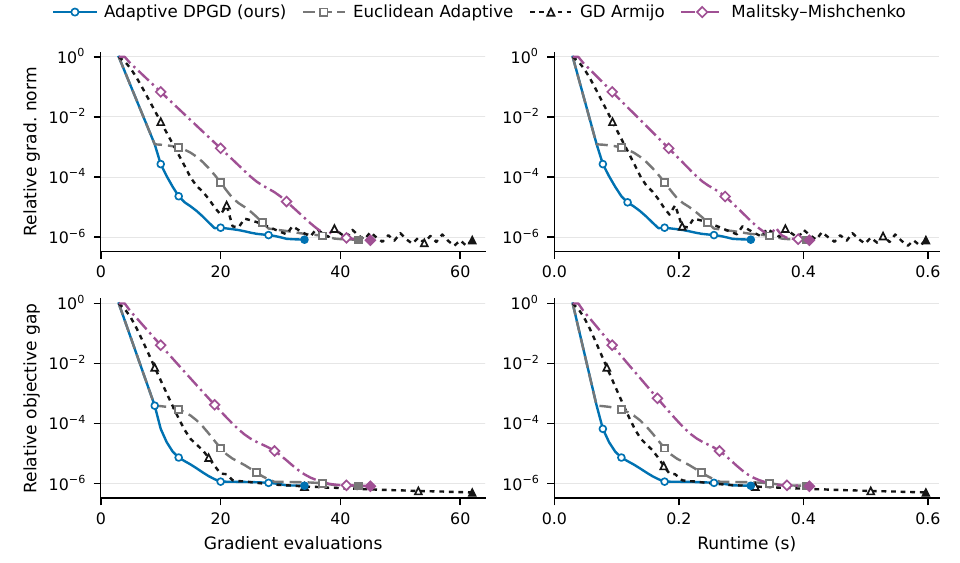}
    \caption{Full comparison for \cref{fig:numerical-main}(c) on \texttt{LEDGAR}. Rows and columns follow \cref{fig:main-full-covtype}.}
    \label{fig:main-full-ledgar}
\end{figure}

\begin{figure}[!htb]
    \centering
    \includegraphics[width=\linewidth]{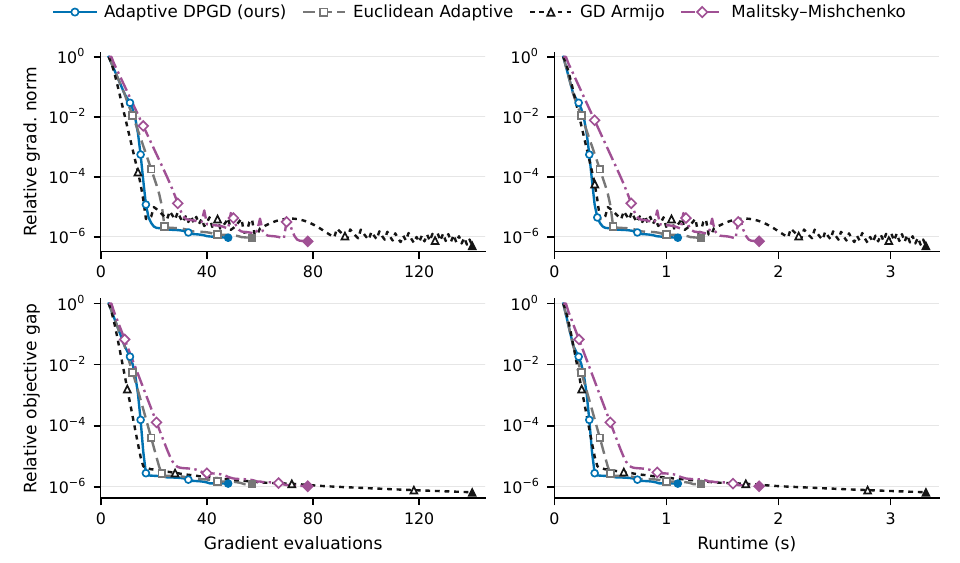}
    \caption{Full comparison for \cref{fig:numerical-main}(d) on \texttt{SCOTUS}. Rows and columns follow \cref{fig:main-full-covtype}.}
    \label{fig:main-full-scotus}
\end{figure}

\clearpage

\begin{figure}[!htb]
    \centering
    \includegraphics[width=\linewidth]{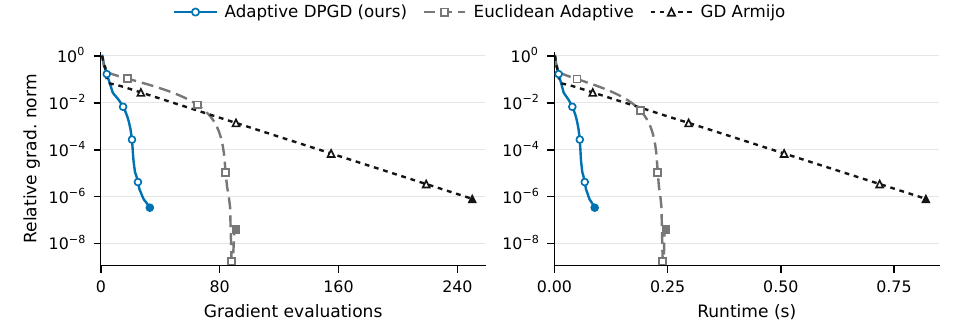}
    \caption{Full comparison for \cref{fig:numerical-main}(e) on \texttt{Skin Non-Skin}. Relative gradient norms against gradient evaluations (left) and runtime (right).}
    \label{fig:main-full-skin_nonskin}
\end{figure}

\begin{figure}[!htb]
    \centering
    \includegraphics[width=\linewidth]{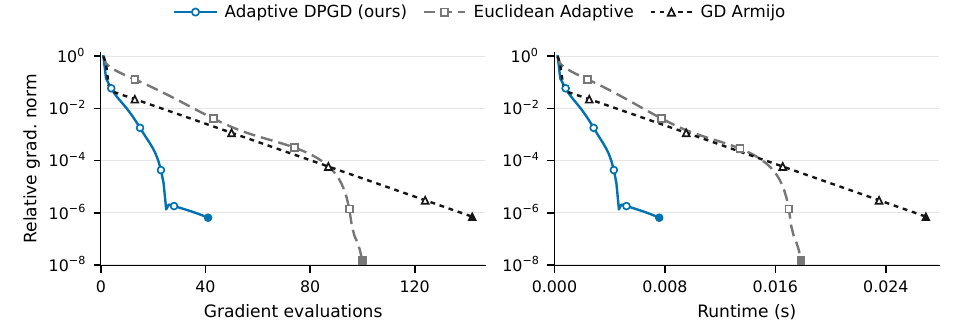}
    \caption{Full comparison for \cref{fig:numerical-main}(f) on \texttt{Rice Cammeo/Osmancik}. Relative gradient norms against gradient evaluations (left) and runtime (right).}
    \label{fig:main-full-rice_cammeo}
\end{figure}

\clearpage
\section{Extended Related Work}\label[appendix]{app:extended-related-work}\label[appendix]{app:comparison-notes}\label[appendix]{app:preconditioning-related-work}
\subsection{Non-Euclidean methods}
Mirror descent~\citep{Nemirovski1983-yw} provides a classical framework for non-Euclidean geometry. Reference functions allow Bregman proximal gradient methods to control first-order model errors beyond global Lipschitz gradient continuity, with guarantees for convex~\citep{Bauschke2017-hg}, nonconvex~\citep{Bolte2018-zt}, and difference-of-convex problems~\citep{Takahashi2022-ml}. Quadratic approximations of Bregman divergences further simplify proximal subproblems~\citep{Takahashi2024-ej,Fujiki2025-do}.
For matrix factorization, reference functions tailored to the loss yield tractable simultaneous updates, including inertial Bregman proximal gradient methods~\citep{Mukkamala2019-mk} and majorization--minimization for nonnegative factorization with Kullback--Leibler loss~\citep{Takahashi2026-rv}. Bregman geometry also guides adaptive Frank--Wolfe stepsizes for weakly convex objectives~\citep{Takahashi2026-ch}.

For Legendre convex $f$, DPGD~\citep{maddison2021} can be viewed as mirror descent on $\phi$ in gradient space, with $f^*$ as the reference function. Here, the superscript $*$ denotes convex conjugation. It also relates preconditioning to inverse objective curvature: global dual relative smoothness compares $D_\phi(\nabla f(x),\nabla f(y))$ with $D_f(y,x)$. Its examples for $p$-norm regression and exponential penalties correct nonquadratic growth while remaining quadratic near the origin. Their guarantees require a global comparison for each objective--preconditioner pair. Our recipe makes this curvature-based construction systematic, while the local conditions in \cref{prop:power-interpolation-preconditioner} permit the resulting preconditioners without a problem-specific global Hessian bound.

\citet{oikonomidis2025nonlinearly,oikonomidis2026nonlinearly} unify normalization, clipping, and memoryless AdaGrad/Adam variants using a reference function $\varphi$, with $\nabla\varphi^*$ corresponding to our $\nabla\phi$. They motivate HGD through a nonlinear global upper bound based on $\varphi(y)=\cosh(\|y\|)-1$. Our construction instead derives the same preconditioner from the normal inverse curvature of an exponential model, identifying the curvature it corrects and extending the procedure to other targets (\cref{tab:preconditioner-recipes}).

The DPGD and nonlinear preconditioning guarantees above use global relative, anisotropic, or generalized smoothness conditions. DPGD linesearch also evaluates gradients at trial points. Adaptive DPGD uses local conditions and updates the stepsize directly from local curvature estimates, without a linesearch after initialization.

\subsection{Further adaptive methods}
Beyond the methods discussed in the main text, \citet{yagishita2025simplelinesearchfreefirstordermethods} study auto-conditioning based on local curvature estimates, and \citet{HoaiThai2024} give a linesearch-free nonconvex proximal-gradient rule under additional structural assumptions. \citet{PatelDiminishing24} analyze nonconvex GD with diminishing stepsizes without global Lipschitz smoothness. Our focus is an explicit nonlinearly preconditioned update with adaptive stepsizes under local regularity and known weak convexity.

For convex optimization, \citet{Li2025-cb} develop a linesearch-free auto-conditioned fast gradient method with uniformly optimal rates across smoothness classes, while \citet{Suh2025-eq} propose a parameter-free adaptive Nesterov method without linesearch. \citet{Takahashi2026-tp} combines adaptive stepsizes with conditional gradient sliding for projection- and linesearch-free acceleration. \citet{Park2026-cd} extend linesearch-free adaptive gradient descent to Riemannian manifolds, with applications to Gaussian variational inference.

Adaptive accelerated mirror descent~\citep{xuChenAAMD2026} combines mirror descent and DPGD under global dual relative smoothness and additional curvature conditions, with backtracking when needed. Accelerating our local, linesearch-free methods remains a separate question.

\end{document}